\documentclass[11pt,a4paper]{amsart}
\usepackage[T1]{fontenc}
\usepackage{lmodern}
\usepackage{amsmath,amssymb,amsthm,mathtools}
\usepackage[a4paper,left=27mm,right=27mm,top=27mm,bottom=28mm]{geometry}
\usepackage{microtype,enumitem,booktabs,xcolor}
\usepackage[colorlinks=true,linkcolor=blue!45!black,citecolor=blue!45!black,urlcolor=blue!45!black]{hyperref}
\hypersetup{pdftitle={Stable Allen--Cahn solutions in three dimensions},pdfauthor={}}
\newcommand{\R}{\mathbb R}
\DeclareMathOperator{\supp}{supp}
\DeclareMathOperator{\dist}{dist}
\DeclareMathOperator{\Area}{Area}
\DeclareMathOperator{\divg}{div}
\DeclareMathOperator{\arctanh}{arctanh}
\DeclareMathOperator{\sgn}{sgn}

\DeclareMathOperator{\sech}{sech}
\DeclareMathOperator{\Hess}{Hess}
\DeclareMathOperator{\Length}{Length}
\newtheorem{theorem}{Theorem}[section]
\newtheorem{lemma}[theorem]{Lemma}
\newtheorem{proposition}[theorem]{Proposition}
\newtheorem{corollary}[theorem]{Corollary}
\theoremstyle{remark}

\numberwithin{equation}{section}
\allowdisplaybreaks[1]
\setlist{topsep=4pt,itemsep=2pt}

\begin{document}
\title{Stable De Giorgi conjecture of the Allen--Cahn equation in $\mathbb{R}^3$}

\author[Y. Liu]{Yong Liu}
\address{School of Mathematics and Statistics, Beijing Technology and Business University, Beijing, China}
\email{yliumath@btbu.edu.cn}

\author[T. Luo]{Tianci Luo}
\address{School of Mathematical Sciences, University of Science and Technology of China, Hefei 230026, P.R. China}
\email{Luo\_tianci@mail.ustc.edu.cn}

\author[K. Wang]{Kelei Wang}
\address{School of Mathematics and Statistics, Wuhan University, Wuhan, China}
\email{wangkelei@whu.edu.cn}

\author[J.C. Wei]{Juncheng Wei}
\address{Juncheng Wei, Department of Mathematics, Chinese University of Hong Kong, Shatin, New Territories, Hong Kong}
\email{wei@math.cuhk.edu.hk}

\author[Y. Wei]{Yong Wei}
\address{School of Mathematical Sciences, University of Science and Technology of China, Hefei 230026, P.R. China}
\email{yongwei@ustc.edu.cn}

 \author[K. Wu]{Ke Wu}
\address{School of Mathematics, Yunnan Normal University, Kunming, China}
\email{kewu@ynnu.edu.cn}

\begin{abstract}
We prove that every bounded stable entire solution $v$ of the Allen--Cahn equation  in $\R^3$ is one-dimensional.
As a consequence, the full De Giorgi conjecture in $\mathbb{R}^4$ is true.
We also obtain local curvature estimates for stable solutions. The proof strategy is inspired by the recent breakthrough work of Chan, Fern\'{a}ndez-Real, Figalli and Serra [J. Amer. Math. Soc. 2026], by reducing the stabilty condition for the Allen-Cahn equation to a weak stability condition on a surface (the zero set) and then utilizing Gauss-Bonnet formula. For this purpose, we first use
the stability condition to get a sublinear bound for a weighted integral
that controls the zeros where the solution is far from planar.
If such zeros exist, we isolate a bounded set of them and join
$1-v^2$ near this set to derivatives of one-dimensional
transitions farther away.
By controlling the interaction between these transitions, we derive the weak stability condition on the zero set, which is then used to bound a weighted integral of the
squared curvature on the regular part of the zero set
by a cutoff gradient integral and a controlled error.
We use this inequality to bound the intrinsic area and construct
logarithmic cutoffs.
The resulting compactly supported test function has a negative
contribution near this set that exceeds all joining and
cutoff errors, contradicting stability.
\end{abstract}

\maketitle

\section{Introduction}

We study bounded entire solutions of the Allen--Cahn equation
\begin{equation}\label{aceq}
-\Delta v=v-v^3 \ \ \quad \mbox{in}~ \mathbb{R}^n.
\end{equation}
The associated double well potential and energy functional are
\[
W(s)=\frac14(1-s^2)^2,
\qquad
E(v;\Omega)=\int_\Omega
\left(\frac12|\nabla v|^2+W(v)\right)\,dx.
\]
A solution $v$ in an open set $\Omega$ is stable when
\[
Q_v(f):=\int_\Omega
\bigl(|\nabla f|^2+(3v^2-1)f^2\bigr)\,dx\geq0
\qquad\text{for every }f\in C_c^1(\Omega).
\] Note that stability also implies the existence of a positive kernel for its linearized operator $L_v$. We are interested here mainly in the classification of stable solutions to the Allen-Cahn equation in $ \mathbb{R}^n$.

If no restriction is imposed on the solution, the Allen-Cahn equation may have many different entire solutions. Among them, the simplest nonconstant solution is the heteroclinic solution $\tanh(x_n/\sqrt2)$.
It connects the values $-1$ and $1$, and its level sets are parallel
hyperplanes. Note that a solution
which is strictly increasing in one direction has a positive derivative
solving the linearized equation, which implies the stability inequality. Hence the one dimensional heteroclinic solution is automatically stable.

The famous De Giorgi conjecture concerns whether bounded solutions that are strictly
increasing in one direction are one-dimensional in dimensions at most
eight \cite{DG79}. Ghoussoub and Gui \cite{GG98} prove the two-dimensional case, and Ambrosio and Cabr\'e \cite{AC00} prove the three-dimensional case. Alberti, Ambrosio and Cabr\'e \cite{AAC01} treat more general nonlinearities
in three dimensions and examine the minimizing property of monotone
solutions. The Liouville theorems of Barlow, Bass and Gui \cite{BBG00}
explain how growth estimates enter this symmetry argument.
Ghoussoub and Gui \cite{GG03} also obtain results under quantitative assumptions
on the energy at infinity, with applications in dimensions four and
five.

In higher dimensions, Savin proves that global minimizers are
one-dimensional for $n\leq7$ \cite[Theorem~2.3]{Sav09}.
In particular, he proves the De Giorgi conjecture for $n\leq8$ when the two
directional limits are $-1$ and $1$ \cite[Theorem~2.4]{Sav09}. We also mention that Wang \cite{W17} develops another proof based on harmonic approximation and
improvement of flatness.
The dimension restriction reflects the geometry of minimal hypersurfaces. In this respect, we point out that
Simons' work \cite{Sim68} plays an important role in the analysis of critical dimension for stable minimal cones, and Bombieri, De Giorgi and Giusti \cite{BDGG69} construct nonaffine
entire minimal graphs over $\R^8$.
Using such graphs, del Pino, Kowalczyk and Wei
\cite{dPKW11} construct counterexamples
to the De Giorgi conjecture in every dimension at least nine.  Under the assumption of uniform convergence of the limits to $\pm 1$—a statement known as Gibbons' Conjecture—one-dimensional symmetry is proved by Barlow, Bass, and Gui \cite{BBG00} as well as Berestycki, Hamel, and Monneau \cite{BHM00} using different techniques.

Dancer formulated a broader classification conjecture for bounded
stable solutions of semilinear equations in his 2010 ICM lecture
\cite[Section~1, p.~1902]{Dan10}.
His original formulation included dimensions up to eight.
For the Allen--Cahn equation, Pacard and Wei \cite[Section~1]{PW13} explicitly stated
the conjecture in dimensions $n\leq7$. Pacard and Wei \cite{PW13} indeed construct bounded stable entire
solutions with zero sets asymptotic to minimal cones in every dimension
at least eight. Later,
Liu, Wang and Wei \cite{LWW17} construct non-one-dimensional global minimizers in
the same dimension range.
The latter result shows that the dimension range in Savin's classification
of minimizers is sharp. It also realizes the connection between
lower-dimensional minimizers and higher-dimensional monotone solutions
studied by Jerison and Monneau \cite{JM01,JM04}. We also would like to point out that Cabr\'e \cite{Cab12,CT09} proved that the saddle-shaped solution
of the Allen--Cahn equation is stable in every even
dimension $n\geq14$ while Liu, Wang and Wei \cite{LWW1, LWW2} proved its stability in dimensions $n=8, 10, 12$.
These solutions have the Simons cone as their zero set
and provide examples of bounded stable solutions
that depend on more than one variable.
These results distinguish the classification of stable solutions from
the more restrictive monotone problem.

The connection with minimal hypersurfaces is also visible in limits of
rescaled solutions. Modica \cite{Mod87} identifies the limiting area problem for
mass-constrained energy minimizers, and Caffarelli and C\'ordoba
\cite{CC95} prove uniform
convergence of their transition sets under the corresponding hypotheses.
Hutchinson and Tonegawa
\cite{HT00} treat critical points with bounded rescaled
energy and obtain limiting interfaces with integer multiplicity. This multiplicity allows several transition layers to
approach the same limiting hypersurface.
In the opposite direction, Kohn and Sternberg \cite{KS89} use
$\Gamma$-convergence to construct local minimizers of the
Allen--Cahn energy near isolated $L^1$ local minimizers
of the perimeter functional. Pacard and Ritor\'e \cite{PR03} construct solutions near
suitably nondegenerate minimal hypersurfaces, and also treat constant mean curvature
hypersurfaces under a volume constraint.
Guaraco \cite{G18} uses a min--max construction for the Allen--Cahn energy to
produce minimal hypersurfaces in closed manifolds. These hypersurfaces
are smooth and embedded in ambient dimensions at most seven.

Stability property also contains geometric information.
Sternberg and Zumbrun \cite{SZ98} use a geometric form of the second variation
in their study of mass-constrained phase boundaries in strictly
convex domains. Farina, Sciunzi and Valdinoci \cite{FSV08} developed a
geometric approach to one-dimensional symmetry based on
the stability inequality.
Their main estimate controls weighted integrals of
level-set curvature and tangential derivatives of
$|\nabla v|$ by an integral involving the gradient of
a cutoff function.
For the Allen--Cahn equation, this approach yields
one-dimensional symmetry for bounded stable solutions
in $\R^2$ and bounded solutions that are strictly
monotone in one direction in $\R^3$.
Tonegawa
\cite{T05} proves a weak form of the geometric stability inequality
for limiting interfaces, and Tonegawa and Wickramasekera
\cite{TW12} establish regularity of
their support under local bounds for the solution and its energy.
These results are closely related to the regularity theories for stable
minimal hypersurfaces developed by Schoen and Simon \cite{SS81} and
Wickramasekera \cite{Wic14}.
The energy hypothesis is important here because it bounds the area
of the limiting interface. For an arbitrary stable entire solution,
such a growth estimate must first be established.

We also recall that for the fractional Allen--Cahn equation, Cabr\'e and Cinti \cite{CC10} established energy estimates. They used these
estimates to prove that bounded global minimizers and
bounded solutions that are strictly monotone in one
direction are one-dimensional in $\R^3$.
They later extended the energy estimates and the
three-dimensional symmetry result to the operators
$(-\Delta)^s$ with $1/2\leq s<1$ \cite{CC14}. Figalli and Serra
\cite{FS20} proved that every bounded stable entire solution
of $(-\Delta)^{1/2}v=v-v^3$ in $\R^3$ is one-dimensional. In the fractional case, nonlocal effects introduce strong interactions.

A parallel problem for minimal surfaces is the stable Bernstein problem: must a complete, connected, two-sided stable minimal
immersion be an affine hyperplane?
Here two-sided means that a global unit normal can be chosen.
Schoen, Simon and Yau \cite[Theorem~3]{SSY75} obtain curvature estimates for hypersurfaces of
dimension at most five under a bound for normalized hypersurface volume.
In $\R^3$, do Carmo and Peng \cite{dCP79}, and independently Fischer-Colbrie and
Schoen \cite{FCS80}, prove that complete two-sided stable minimal surfaces are
planes without an area-growth assumption.
Thus even in the minimal-surface problem, a classification under a
growth bound and a classification under stability alone require
different arguments.

Recent work extends this classification to higher dimensions without
a volume-growth assumption. Chodosh and Li \cite{CL24} proved the theorem in
$\mathbb{R}^{4}$. Their second proof \cite{CL23} combines
the Gulliver--Lawson conformal change with warped $\mu$-bubbles to
obtain intrinsic cubic volume growth on the simply connected cover,
after which curvature estimates give flatness. This approach is
further developed by Chodosh, Li, Minter and Stryker \cite{CLMS26} in
$\mathbb{R}^{5}$, and by Mazet
\cite{Maz24} in $\mathbb{R}^{6}$, using spectral curvature estimates and volume comparison
to control the auxiliary hypersurfaces.

In a different direction, Cabr\'e, Catino, Mari, Mastrolia and
Roncoroni \cite{CCMMR26} give another proof in $\mathbb{R}^{4}$
based on weighted curvature inequalities and sharp gradient estimates
for the Green kernel under spectral Ricci lower bounds. Their
argument uses cutoff functions constructed from the Green kernel
and does not involve $\mu$-bubbles. More recently, in a preprint,
Hong, Li and Wang \cite{HLW26} present a Green-function proof of the
stable Bernstein theorem for smooth, connected, complete, two-sided
stable minimal immersions $M^{6}\to\mathbb{R}^{7}$, the last remaining
case of the problem. Their argument combines Green-function
identities and stability with the Codazzi equations and additional
tensorial divergence identities to obtain critical integral estimates
and deduce flatness.

For the Allen--Cahn equation, the growth assumption is usually expressed in terms of  bounded energy density:
\[
\sup_{R\geq1}R^{1-n}E(v;B_R)<\infty.
\]
This means quadratic energy growth in $\R^3$ and cubic growth in
$\R^4$. Boundedness of the solution and interior elliptic estimates
give only a bound proportional to the volume of the ball, which is
of order $R^3$ in three dimensions.
A one dimensional heteroclinic solution instead has energy of order $R^2$, since its
transition region has bounded width around a plane.
The difference between these bounds is precisely the difficulty
that a proof under stability alone must address.

In $\R^2$, the argument of Ghoussoub and Gui as well as  Ambrosio and Cabr\'e yields the
one-dimensional classification of stable solutions. (See an earlier argument of Berestycki, Caffarelli and Nirenberg \cite{BCN1997}.)
Stability supplies the positive solution of the linearized equation
needed in this argument \cite{GG98,AAC01}.
In $\R^3$, quadratic energy growth permits the one-dimensional
conclusion by the arguments of Ambrosio and Cabr\'e and Alberti,
Ambrosio and Cabr\'e \cite{AC00,AAC01}.
Florit-Simon and Serra state this result as
\cite[Theorem~2.10]{FSS25}.
They prove the corresponding stable classification in $\R^4$
under cubic energy growth, together with local curvature estimates
under an energy bound \cite[Theorems~1.3 and~1.4]{FSS25}.
Florit-Simon \cite{FS26} subsequently proves that finite Morse index and bounded
energy density imply one-dimensionality in $\R^4$.

Curvature estimates also connect stability to the behavior at infinity.
Wang and Wei \cite{WWends} prove that finite Morse index implies finitely many ends
and linear energy growth in $\R^2$.
Their proof derives curvature decay by describing the interaction of
nearby transition layers through the Toda system.
In their later work, they \cite{WW19} establish second-order estimates and
quantitative separation for stable solutions in higher dimensions
under local geometric hypotheses.
Chodosh and Mantoulidis
\cite{CM20} obtain related curvature and separation
estimates, and study index and multiplicity, on three-manifolds.
The local estimates of Wang and Wei are an essential part of the
proof below. The remaining task is to use them in a global argument
that starts with stability and boundedness alone.

Our first theorem completes this argument and proves the stable solution conjecture in $\mathbb{R}^3$.
\begin{theorem}\label{stable}
Let $v$ be a bounded stable solution of \eqref{aceq} in $\R^3$. Then either $v\equiv1$, $v\equiv-1$, or
\begin{equation}\label{planar}
v(x)=\tanh\left(\frac{e\cdot x-a}{\sqrt2}\right),
\qquad e\in\mathbb S^2,\quad a\in\R.
\end{equation}
\end{theorem}

Theorem~\ref{stable} also determines the two directional limits of a
bounded monotone solution in $\R^4$.
Each limit is a stable solution in the three transverse variables,
so it is constant or a planar transition.
Both types are global minimizers. The theorem of Jerison and Monneau \cite[Theorem~1.3]{JM04}
on minimizing directional limits then makes the original solution a
global minimizer.
Savin's classification of minimizers applies in dimension four
\cite[Theorem~2.3]{Sav09}, yielding the following consequence.
\begin{theorem}\label{monotone}
Let $u:\R^4\to(-1,1)$ solve \eqref{aceq} and satisfy $u_{x_4}>0$ everywhere. Then$$
u(x)=\tanh\left(\frac{e\cdot x-a}{\sqrt2}\right),
\qquad e\in\mathbb S^3,\quad e_4>0,\quad a\in\R.$$
\end{theorem}

The entire classification result we proved here has a local version as well. Observe that compactness shows that a stable solution on a sufficiently large ball
is close to a planar transition near each zero on balls of fixed radius.
This gives the local approximation required by Wang and Wei's
curvature estimates \cite[Corollary~1.3 and Theorem~1.1]{WW19}.
In the following statement, $A$ is the second fundamental form of a
level set and $H$ is its mean curvature. The stronger decay of $H$
comes from the Allen--Cahn equation and the estimates in that paper.
\begin{corollary}\label{localcurv}
For every $\tau\in(0,1)$ there are constants $R_\tau\geq1$ and $C_\tau>0$ with the following property. If $R\geq R_\tau$ and $v:B_{4R}\subset\R^3\to(-1,1)$ is a stable solution of \eqref{aceq}, then $\nabla v\neq0$ on $B_R\cap\{|v|\leq\tau\}$. The second fundamental form and mean curvature of its level sets satisfy
\[
|A|\leq\frac{C_\tau}{R},
\qquad
|H|\leq\frac{C_\tau}{R^2}
\qquad\text{on }B_R\cap\{|v|\leq\tau\}.
\]
\end{corollary}

The methods developed for the stable Bernstein problem
\cite{CCMMR26,CL23,HLW26} provide a guiding perspective for our
global argument. In different ways, they derive from stability
the geometric or weighted integral estimates needed to control
cutoff errors, rather than assume a growth bound at infinity.
In our setting, the local transition-layer estimates allow us
to transfer stability to a curvature inequality on the regular
zero set. We then use this inequality to obtain intrinsic area
bounds and logarithmic cutoffs, replacing an a priori energy-growth
assumption in the global argument. The implementation is different:
we use neither $\mu$-bubbles nor Green functions, and must control
the interactions between transition layers when extending the
surface cutoffs to a single test function in $\mathbb{R}^{3}$.

More directly, our argument is closely related to the proof for
the free boundary Allen--Cahn problem by Chan, Fern\'andez-Real, Figalli and Serra  in  \cite[Section~10]{CFFS25}. Indeed, both arguments combine stability with
the geometry of level sets, the Gauss--Bonnet type formula, and logarithmic cutoff functions. For the smooth Allen--Cahn equation, the gradient becomes small as
the solution approaches either of its limiting values, and nearby
transition layers interact. We use $1-v^2$ near the selected set to
obtain a definite negative contribution to the second variation.
Farther away, we use derivatives of one-dimensional transitions to
derive a curvature inequality on the zero set. The additional work
is to control the mixed terms between these derivatives and the
errors produced when we combine the inner and outer test functions.
We also use a different geometric construction for the area estimate.
Lemma~10.16 of \cite{CFFS25} applies Gauss--Bonnet directly to regions defined
by intrinsic distance from a compact set. In Lemma~\ref{cap}, we
attach an auxiliary surface along all inner boundary curves,
preserving the metric on the original surface. Every point of
these curves can then be reached from one common point along a
path of uniformly bounded length. We control the area of the finite
added part and the total absolute Gaussian curvature. Together
with the extension of the surface inequality, this construction
yields quadratic intrinsic area bounds and logarithmic cutoff
estimates. The cutoff costs remain linear in the number of covering
balls, with coefficients independent of the diameter of the selected
set, allowing comparison with the negative contribution near that set.
Thus, the geometric strategy is similar, while the passage from
stability to surface estimates and the final construction require
estimates adapted to smooth transitions.

Let us explain our argument in more details. For a nonconstant solution with values in $(-1,1)$, we set
\[
h=1-v^2,\qquad
\theta=\sqrt2\,\arctanh v,\qquad
r=1-|\nabla\theta|^2.
\]
Modica estimate implies $0\leq r\leq1$ \cite{Mod85}.
For every compactly supported cutoff $\eta$, the equation gives
\[
Q_v(h\eta)=\int h^2|\nabla\eta|^2-\int h^3r\eta^2.
\]
Thus stability controls the nonnegative integral of $h^3r$ through
the cost of cutting off $h$. This identity is the starting point for the global estimates.

A small value of $r$ at a zero yields closeness to the one dimensional solution
on a ball whose radius grows like $\log(1/r)$.
We use these balls to cover the zeros where $r$ is small, while the
integral of $h^3r$ controls the other zeros.
For every $x\in\R^3$ and $R\geq e$, we obtain
\[
E(v;B_R(x))+\int_{B_R(x)}h\leq\frac{CR^3}{\log R},
\qquad
\int_{B_R(x)}h^3r\leq\frac{CR}{\log R}.
\]
The first estimate still permits growth larger than $R^2$.
The second estimate  allows us to isolate a
bounded group of zeros where $r$ exceeds a fixed small positive number.
We call these zeros nonflat.
An unbounded connected chain of such zeros would force at least linear
growth of the same integral, contradicting its sublinear upper bound.
Every connected component of a fixed finite-width neighbourhood of
the nonflat zeros is therefore bounded.

If nonflat zeros exist, we choose one such component and denote its
nonflat zeros by $K$.
We cover $K$ using doubled balls from a maximal collection of $N$
pairwise disjoint unit balls with centres in $K$.
On each unit ball, the integral of $h^3r$ has a fixed positive lower
bound. Their total is at least $b_*N$ for a constant $b_*>0$.
Our aim is to construct a compactly supported test function that
equals $h$ near $K$. Its negative contribution there must exceed
the positive terms produced farther away.
We therefore estimate every additional term by the same factor $N$
times a coefficient that can be made small.

In the region separating $K$ from the other nonflat zeros, the local
approximation permits us to apply Wang and Wei's estimates
\cite{WW19}.
Their Corollary~1.3 uses the classical stable minimal-surface theorem
in $\R^3$.
The estimates give quantitative separation of nearby transitions and
a $C^1$ approximation with error bounded by $Cd^{-2}$, where $d$ is
the distance to the nonflat zeros.
The constants are uniform in the number of local transition layers.
We use the regularized distance from \cite[Chapter~VI, Theorem~2]{Stein70}
when derivatives of a cutoff are needed.

Outside $K$, our test function is a sum of derivatives of
one-dimensional transitions, taken in directions perpendicular to the
regular zero surface and multiplied by cutoffs.
Several terms can contribute at the same point, so their products
must also be included in the second variation.
We use Modica's estimate to control the potential terms and the exact
change-of-variables formula to compute the curvature terms.
After integration, the latter include a negative multiple of $|A|^2$.
We estimate the products between different transitions before passing
to an inequality on the zero surface.
This is the additional calculation required by smooth transitions,
whose derivatives overlap even when their zero surfaces are disjoint.

The curvature inequality \eqref{surfstab}, obtained from the
stability of $v$, can be rewritten using the Gauss equation
as \eqref{curvform}. In this form, the Gaussian curvature
appears with coefficient $2/3$, while the remaining error
is represented by the nonnegative function $W_Y$.
We use it to control intrinsic area on the part of the surface outside
$K$, where intrinsic distance means distance measured along the surface.
An auxiliary surface joins the inner boundary circles, and its
contribution to the inequality is bounded by a constant times $N$.
The inequality of B\'erard and Castillon for functions of intrinsic
distance then supplies the required area bound
\cite[Lemma~2.3, equation~(2.12)]{BC14}.
This bound allows logarithmic cutoff functions with small Dirichlet
energy.
The final test function is defined on the original zero surface and
then on $\R^3$.

We finally join the exterior test function to $h$ near $K$.
The $C^1$ approximation controls the error in the joining region,
including the term created when a nonconstant weight is integrated
by parts.
Stability then yields
\[
b_*\leq\frac{C_R}{\log T}
+C\left(\frac1{\log R}+\frac{\log R}{R}+R^{-2}\right)
\]
after division by $N$.
Here $R$ determines the region where the functions are joined and
$T$ determines the extent of the exterior cutoff.
We choose $R$ first and then $T$, making the right-hand side smaller
than the fixed positive number $b_*$.
The constants are independent of the resulting $N$ and the diameter
of $K$, so the construction leads to a contradiction.
This argument combines the sublinear integral estimate in three
dimensions with logarithmic cutoffs on a two-dimensional surface.

This paper is organized as follows. Section~2 proves the preliminary estimates and the local approximation
away from nonflat zeros. Section~3 constructs the test functions on
$\R^3$ and derives the surface inequality. Section~4 constructs the
exterior cutoffs and estimates their cost. Section~5 proves the stable
classification and its monotone and local curvature consequences.

\section{Preliminary estimates and regions close to a plane}\label{prelims}

In this section, we do some preliminary analysis for the stable solution. To begin with, let us state the following elementary result.
\begin{lemma}\label{phaseid}
Let $v$ be a bounded entire solution  with $-1<v<1$ and  $h,\theta,r$ be the functions introduced in Section 1. Then the following properties hold.
\begin{enumerate}
    \item
$0\leq r\leq 1$, and
\begin{equation}\label{phaseeq}
\begin{gathered}
 \nabla v=\frac{h}{\sqrt2}\nabla\theta,\qquad
 \Delta\theta=-\sqrt2\,vr,\\
 \bigl(-\Delta+2\sqrt2\,v\nabla\theta\cdot\nabla
                  +2h|\nabla\theta|^2\bigr)r
       =2|D^2\theta|^2,\qquad L_vh=-h^2r.
\end{gathered}
\end{equation}
\item For every positive integer $j$, we have
$|D^jv|\leq C_jh$ and $|D^j\theta|\leq C_j$.

\item If $r$ vanishes at a point, then $v$ has the form \eqref{planar}.

\item If $v$ is stable, then
\begin{equation}\label{weightest}
 \int_{\R^3}h^3r\eta^2\leq\int_{\R^3}h^2|\nabla\eta|^2
 \qquad\text{for every }\eta\in H^1_c(\R^3).
\end{equation}
\end{enumerate}
\end{lemma}

\begin{proof}

We give the proof for completeness.

Modica's gradient estimate \cite{Mod85} states that
$|\nabla v|^2\leq h^2/2$. In particular,
\[
 h^3r=(1-v^2)\bigl((1-v^2)^2-2|\nabla v|^2\bigr)\geq0.
\]
For the constant solutions $1$ and $-1$, we interpret every integral
of $h^3r$ below through the polynomial expression on the right, whose
value is zero. Substitution of
$v=\tanh(\theta/\sqrt2)$ into \eqref{aceq}, followed by differentiation of
$|\nabla\theta|^2$, proves \eqref{phaseeq}. We also have
\[
 \Delta(1+v)=v(v-1)(1+v),\qquad
 \Delta(1-v)=v(v+1)(1-v).
\]
Interior estimates first bound every fixed order of derivative of $v$
uniformly. The coefficients in these two linear equations consequently have
uniform bounds of every fixed order. Harnack's inequality and interior
estimates on fixed balls imply
\[
 |D^jv|\leq C_j\min(1-v,1+v)\leq C_jh.
\]
The chain rule yields the estimates for $\theta$. In particular, $r$ and
$D^2\theta$ are uniformly Lipschitz, and
\[
 |\nabla\log h|\leq\sqrt2.
\]
For the last equation in \eqref{phaseeq}, we compute directly
\[
 \Delta h=-2|\nabla v|^2+2v^2h,
 \qquad L_vh=2|\nabla v|^2-h^2=-h^2r.
\]
For the positive function $h$, integration by parts in the terms
containing $\eta^2|\nabla h|^2+2h\eta\nabla h\cdot\nabla\eta$
then gives
\[
 Q_v(h\eta)=\int_{\R^3}h^2|\nabla\eta|^2
             +\int_{\R^3}h\eta^2L_vh.
\]
 Substituting $L_vh$ into the above equation and applying stability
 proves \eqref{weightest} for smooth test functions. Finally, $r$ is a nonnegative supersolution of the
operator in \eqref{phaseeq}, whose zeroth-order coefficient is nonnegative.
The strong minimum principle implies $r\equiv0$ if $r$ vanishes at a point.
The same equation then yields $D^2\theta=0$. Thus
$\theta=e\cdot x-a$ with $|e|=1$, as required.
\end{proof}

The next lemma shows how a small value of $r$ at a zero
leads to an energy bound in a surrounding ball.

\begin{lemma}\label{flatball}
There are constants $c_0>0$ and $\varepsilon_0>0$ such that, if
$v(y)=0$ and $r(y)\leq\varepsilon$ with
$0<\varepsilon<\varepsilon_0$, then
\begin{equation}\label{flatenergy}
 E(v;B_s(y))\leq Cs^2
 \qquad\text{for }0<s\leq c_0\log(1/\varepsilon).
\end{equation}
\end{lemma}

\begin{proof}
We use the drift vector $b=2\sqrt2\,v\nabla\theta$ in
\eqref{phaseeq}. Both $b$ and the nonnegative coefficient
$2h|\nabla\theta|^2$ are uniformly bounded, and
\[
 \divg b=2h|\nabla\theta|^2-4v^2r.
\]
The weak Harnack inequality for the nonnegative supersolution
$r$ admits exponent one in dimension three. A chain of unit-scale balls with
uniformly overlapping interiors therefore yields
\[
 \int_{B_2(z)}r\leq C\exp(C|z-y|)r(y).
\]
To see the dependence on $|z-y|$, we cover the segment from $y$ to $z$
by at most $C(1+|z-y|)$ balls. The weak Harnack estimate on the first ball
bounds its integral by $Cr(y)$, because its inner ball contains $y$.
On each overlap, the infimum is at most the average already bounded at the
previous step. The next weak Harnack estimate therefore multiplies the
bound by at most a fixed constant. After at most $C(1+|z-y|)$ steps, this
product is bounded by $C\exp(C|z-y|)$.

We choose $0\leq\chi\in C^\infty_c(B_2(z))$ equal to one on $B_1(z)$.
Integration of the equation for $r$, together with the formula for
$\divg b$, gives
\[
 2\int\chi|D^2\theta|^2
 =\int r\bigl(-\Delta\chi-b\cdot\nabla\chi+4v^2r\chi\bigr)
 \leq C\int_{B_2(z)}r.
\]
The Lipschitz bound for $D^2\theta$ shows that, when $|D^2\theta(z)|>0$,
its norm stays at least $|D^2\theta(z)|/2$ on a ball of radius
$c|D^2\theta(z)|$. The norm is also uniformly bounded, so we can choose
the fixed constant $c>0$ small enough so that this ball lies in $B_1(z)$.
Its three-dimensional volume yields
\[
 c|D^2\theta(z)|^5\leq\int_{B_1(z)}|D^2\theta|^2
       \leq C\exp(C|z-y|)r(y).
\]
Taking fifth roots and enlarging the fixed constants yields
\[
 |D^2\theta(z)|\leq C\exp(C|z-y|)r(y)^{1/5}.
\]
We choose $c_0$ small and then $\varepsilon_0$ small. With
\[
 \ell_\varepsilon=c_0\log(1/\varepsilon),\qquad
 e=\frac{\nabla\theta(y)}{|\nabla\theta(y)|},
\]
we can arrange
$e^{C\ell_\varepsilon}\varepsilon^{1/5}\leq\varepsilon^{1/10}$
by decreasing $c_0$. We integrate the Hessian along the segment from $y$
to each point in $B_{\ell_\varepsilon}(y)$. Since
$|\nabla\theta(y)|=\sqrt{1-r(y)}$, we obtain
\[
 \partial_e\theta\geq\sqrt{1-\varepsilon}
             -C\ell_\varepsilon\varepsilon^{1/10}
 \geq\frac12\qquad\text{in }B_{\ell_\varepsilon}(y).
\]
On every line parallel to $e$, the change of variable from the line
parameter $t$ to $\theta$ consequently gives
\[
 \int h^2\,dt\leq2\int_\R\sech^4(\theta/\sqrt2)\,d\theta.
\]
The energy density is at most $h^2/2$. Integration over the orthogonal
projection of $B_s(y)$ now proves \eqref{flatenergy}.
\end{proof}

The next estimate concerns points far from the zero set.
\begin{lemma}\label{clearing}
If $B_a(x)\cap\{v=0\}=\varnothing$ and $a$ is sufficiently large, then
\[
 h(x)+\frac12|\nabla v(x)|^2+W(v(x))\leq Ce^{-ca}.
\]
\end{lemma}

\begin{proof}The proof is standard, we omit the details.
\end{proof}

 Near zeros where $r$ is small, Lemma~\ref{flatball} provides
balls whose energy is bounded by a constant times the square
of their radius.
Around zeros where $r$ exceeds a chosen positive number, we choose
small balls on which the integral of $h^3r$ has a positive
lower bound depending on that threshold.
The stability estimate then bounds the number of
disjoint balls of this second type.
Away from the zero set, Lemma~\ref{clearing} gives an
exponentially small energy density.
The next proposition combines these estimates to improve
the energy-growth bound, and then applies stability again
to show that the integral of $h^3r$ grows more slowly
than the radius.
\begin{proposition}
For every stable entire solution of \eqref{aceq}, every $x\in\R^3$, and
every $R\geq e$, we have
\begin{equation}\label{growthest}
 E(v;B_R(x))+\int_{B_R(x)}h\leq\frac{CR^3}{\log R},
 \qquad
 \int_{B_R(x)}h^3r\leq\frac{CR}{\log R}.
\end{equation}
\end{proposition}

\begin{proof}
We first estimate the energy by covering neighbourhoods of the zero set.
The clearing estimate handles the rest of the ball. An equation for $h$
then controls its integral, and a second application of stability improves
the bound for the integral of $h^3r$.

Translation allows us to centre the balls at zero. Let us first take $R$
large. A ball cutoff in \eqref{weightest} yields
$\int_{B_{3R/2}}h^3r\leq CR$. We set
\[
 \varepsilon=R^{-1/5},\qquad
 \ell=\frac{c_0}{10}\log(1/\varepsilon),
\]
and divide the zeros in $B_{R+\ell}$ according to whether
$r\leq\varepsilon$ or $r>\varepsilon$.

We choose a maximal disjoint family of balls of radius $\ell/2$ centred
at zeros in the first set. This family contains at most $CR^3/\ell^3$
balls. The balls with the same centres and radius $2\ell$ cover the
$\ell$-neighbourhood of this set. Lemma~\ref{flatball} bounds their total
energy by $CR^3/\ell$.

At each zero in the second set, uniform Lipschitz bounds imply
$h\geq c$ and $r\geq\varepsilon/2$ on a ball of radius
$a_0\varepsilon$, where $a_0$ is fixed and small. The integral of $h^3r$
over each such ball is at least $c\varepsilon^4$. A maximal disjoint family therefore
contains at most
\[
 CR\varepsilon^{-4}=CR^{9/5}
\]
balls. Their doubled balls cover the second set. The uniform bound on the
energy density consequently bounds the energy of its
$\ell$-neighbourhood by $CR^{9/5}\ell^3$. All the small balls used here lie
in $B_{3R/2}$ for large $R$.

Every remaining point of $B_R$ has distance at least $\ell$ from the
whole zero set. Lemma~\ref{clearing} bounds its total energy contribution
by
\[
 CR^3e^{-c\ell}\leq CR^{3-\beta}
\]
for a fixed $\beta>0$. Both error terms are $o(R^3/\log R)$, which
proves the energy estimate. The identity
\[
 h=\frac12\Delta h+h^2+|\nabla v|^2
\]
allows us to estimate $\int h$ as well. We choose a smooth function
$0\leq\zeta\leq1$, equal to one on $B_R$ and zero outside $B_{2R}$,
with $|\Delta\zeta|\leq C/R^2$. Integration by parts gives
\[
\begin{split}
 \int_{B_R}h
 &\leq \frac12\int h\Delta\zeta
       +\int\zeta(h^2+|\nabla v|^2)\\
 &\leq CR+4E(v;B_{2R}).
\end{split}
\]
We used $h\leq1$ for the first integral and $h^2=4W(v)$ for the second.
Finally, we apply \eqref{weightest} with a cutoff equal to one on $B_R$
and supported in $B_{2R}$, whose gradient is bounded by $C/R$. We obtain
\[
 \int_{B_R}h^3r
 \leq \frac{C}{R^2}\int_{B_{2R}}h^2
 \leq \frac{C}{R^2}E(v;B_{2R})
 \leq \frac{CR}{\log R}.
\] This finishes the proof.
\end{proof}

The sublinear bound for $\int_{B_R}h^3r$ allows us to
separate a bounded set of zeros from the remaining zeros
at which $r$ is bounded below by a fixed positive number.
If such zeros exist, the next proposition shows that we
can first prescribe a finite separation distance and then
choose a nonempty compact set $K$ with this separation.
We select $N$ disjoint unit balls centred in $K$ whose
doubled balls cover $K$.
The integral of $h^3r$ over the unit neighbourhood of $K$
is bounded below by a positive constant times $N$, with
the constant independent of the prescribed distance and
the chosen set.

\begin{proposition}\label{clusters}
Let $v$ be a bounded stable entire solution of \eqref{aceq}.
For any fixed $\delta_*>0$, define
\[
 \Sigma=\{v=0\},
 \qquad
 Z_*=\{y\in\Sigma:r(y)\geq\delta_*\},
\]
and assume that $Z_*$ is nonempty.
There exists $b_*>0$, depending only on $\delta_*$, such that
the following conclusions hold for every finite $L>0$.

Each connected component $\mathcal C$ of
$\bigcup_{z\in Z_*}B_L(z)$ is bounded.
The set $K=Z_*\cap\mathcal C$ is nonempty and compact, and
\begin{equation}\label{clsep}
 \dist(K,Z_*\setminus K)\geq2L.
\end{equation}
Every maximal family of pairwise disjoint unit balls with
centres in $K$ is finite. If its centres are
$z_1,\ldots,z_N$, then
\begin{equation}\label{clmass}
 1\leq N<\infty,
 \qquad
 K\subset\bigcup_{j=1}^N B_2(z_j),
 \qquad
 \int_{\{\dist(X,K)<1\}}h^3r\geq b_*N.
\end{equation}
The set $K$ and the number $N$ may depend on $L$,
whereas $b_*$ is independent of $L$.

Finally, if $d=\dist(\cdot,Z_*)$ and
$d_K=\dist(\cdot,K)$, then
\begin{equation}\label{cldist}
 d(X)=d_K(X)
 \qquad\text{whenever }d_K(X)<L.
\end{equation}
\end{proposition}
\begin{proof}
We first observe that each point of $Z_*$ contributes a fixed
positive amount to the integral of $h^3r$. Indeed, $h=1$ and
$r\geq\delta_*$ on $Z_*$. Uniform local Lipschitz estimates
therefore provide constants $a_*\in(0,1/4)$ and $b_*>0$,
depending only on $\delta_*$, such that
\[
 \int_{B_{a_*}(z)}h^3r\geq b_*
 \qquad\text{for every }z\in Z_*.
\]

We fix $L>0$ and a connected component $\mathcal C$.
Each component contains at least one centre of the balls
defining the union, so $K=Z_*\cap\mathcal C$ is nonempty.
Suppose that $\mathcal C$ is unbounded, and fix $z_0\in K$.
By connectedness, $\mathcal C$ meets every sphere centred at
$z_0$. For each integer $j\geq1$, we choose a point of
$\mathcal C$ at distance $4j(L+1)$ from $z_0$, and then a point
$z_j\in Z_*$ within distance $L$ of it.
The distances $|z_j-z_0|$ of two different selected points
differ by at least $4(L+1)-2L>2$. Thus the balls
$B_{a_*}(z_j)$ are disjoint.
For $R\geq8(L+1)$, the indices satisfying $4j(L+1)\leq R$
provide at least $R/[8(L+1)]$ such balls, all contained in
$B_{R+L+1}(z_0)$. The preceding lower bound and
\eqref{growthest} imply
\[
 \frac{b_*R}{8(L+1)}
 \leq \int_{B_{R+L+1}(z_0)}h^3r
 \leq \frac{C(R+L+1)}{\log(R+L+1)}.
\]
We keep $L$ fixed and divide by $R$. As $R\to\infty$,
the right-hand side tends to zero, whereas the left-hand
side remains positive. This contradiction proves that
$\mathcal C$ is bounded.

We next prove the separation and compactness of $K$.
If $z\in K$ and $w\in Z_*$ satisfy $|z-w|<2L$, then
$B_L(z)$ and $B_L(w)$ intersect. They belong to the same
connected component, so $w\in K$. This proves \eqref{clsep}.
Since $Z_*$ is closed, every limit point of $K$ belongs to
$Z_*$, and \eqref{clsep} forces it to belong to $K$.
Thus $K$ is closed and bounded, hence compact.

We now take a maximal disjoint family
$B_1(z_1),\ldots,B_1(z_N)$ with centres in $K$.
Boundedness of $K$ makes this family finite, and maximality
implies that the balls $B_2(z_j)$ cover $K$.
The smaller balls $B_{a_*}(z_j)$ are disjoint and lie in
$\{X:\dist(X,K)<1\}$. Consequently,
\[
 \int_{\{\dist(X,K)<1\}}h^3r
 \geq \sum_{j=1}^N\int_{B_{a_*}(z_j)}h^3r
 \geq b_*N.
\]
This proves \eqref{clmass}.

Finally, if $d_K(X)<L$, the separation \eqref{clsep} gives
\[
 |X-z|\geq 2L-d_K(X)>d_K(X)
 \qquad\text{for every }z\in Z_*\setminus K.
\]
Thus the distance from $X$ to $Z_*$ equals its distance
to $K$, which proves \eqref{cldist}.
\end{proof}
Small values of $r$ at zeros imply that $v$ is close to a
one-dimensional solution on balls of fixed radius. This allows us
to apply Wang and Wei's local estimates
\cite[Corollary~1.3]{WW19}, which use the classification of stable
minimal surfaces in $\R^3$ \cite{dCP79,FCS80} and require neither
an energy bound nor a bound on the number of zero surfaces.
The next proposition controls the curvature and separation of
these surfaces, and the error in approximating $v$ by a sum of
one-dimensional functions.

We first describe this approximation. In a region where the zero
surfaces $\Sigma_\alpha$ are ordered graphs, we centre a copy of
$q(s)=\tanh(s/\sqrt2)$ on each surface. Here $s$ is the signed
distance to that surface, oriented towards increasing $v$.
We leave each copy unchanged for $|s|\leq4\log R$ and use a
fixed smooth even cutoff, rescaled by $\log R$, to make it equal
to $\sgn s$ for $|s|\geq8\log R$. Away from the zero set, we
define $G_0$ as $\sgn v$ plus the sum of the differences between
these truncated copies and their respective $\sgn s$.
We extend $G_0$ across the zero surfaces by continuity.
Each sum below uses the surfaces in the larger graphical region
specified in the proposition.

\begin{proposition}\label{localest}
There are constants $\delta_*\in(0,1/2)$ and $R_0$ such that the
following holds. Suppose $v$ is a stable entire solution,
$R\geq R_0$, and
\[
 r(y)<\delta_*\qquad
 \text{for every }y\in\{v=0\}\cap B_{4R}(x_0).
\]
Near each zero in $B_R(x_0)$, the nearby zero surfaces
$\Sigma_\alpha$ are ordered graphs over a common disk of radius
$cR$, extending over a larger concentric disk of a fixed multiple
of this radius. In a fixed smaller region, we have
\begin{equation}\label{curvest}
 |A_\alpha|+|\nabla_{\Sigma_\alpha}A_\alpha|\leq CR^{-1},
 \qquad |H_\alpha|\leq CR^{-2},
\end{equation}
and
\begin{equation}\label{tailsum}
 \sum_{\beta\ne\alpha}
 e^{-\sqrt2\dist(y,\Sigma_\beta)}\leq CR^{-2}
 \qquad (y\in\Sigma_\alpha).
\end{equation}
In the same smaller region of $\R^3$, we also have
\begin{equation}\label{gluingest}
 |v-G_0|+|\nabla v-\nabla G_0|\leq CR^{-2}.
\end{equation}
All constants are independent of the number of zero surfaces.
\end{proposition}

In particular, \eqref{tailsum} implies that distinct zero surfaces
are separated by at least $\sqrt2\log R-C$ in the smaller region.

\begin{proof}
We first obtain the fixed-radius approximation required by
\cite[equation~(1.8)]{WW19}. If $r(y_j)\to0$ at a sequence of
zeros, elliptic compactness produces a limit after translation.
The limiting solution and its corresponding function $r$ both
vanish at the origin, so Lemma~\ref{phaseid} identifies the limit
as a planar transition. Consequently, choosing $\delta_*$ small
ensures the required approximation at every zero in the hypothesis.
We rescale $B_{4R}(x_0)$ to the unit ball. The equation then has
the potential $(1-s^2)^2/8$ and parameter
$\varepsilon=(4\sqrt2R)^{-1}$ used in \cite{WW19}.
Stability is preserved, and the stable minimal-surface theorem
in $\R^3$ \cite{dCP79,FCS80} allows us to apply their Corollary~1.3.

We next collect the geometric consequences in the original
coordinates. Corollary~1.3 and Theorem~1.1 of \cite{WW19}
yield $|A_\alpha|\leq C/R$ and $|H_\alpha|\leq C/R^2$.
Their Lemma~2.2 supplies the common graphical region, and
Lemma~3.1 yields $|\nabla_{\Sigma_\alpha}A_\alpha|\leq C/R$.
These are the estimates in \eqref{curvest}.
Proposition~10.1 bounds the largest exponential interaction by
$C\varepsilon^2$, while Lemma~3.6 controls the sum by its largest
term. Returning to our coordinates gives \eqref{tailsum}, with
the exponent $\sqrt2$.

It remains to compare $v$ with $G_0$. We work a distance at least
$cR$ from the boundary of the larger graphical region.
This leaves room for the additional balls of radius
$O((\log R)^2)$ used in \cite[Proposition~6.1]{WW19}.
Their Proposition~4.1 constructs a sum of one-dimensional
functions that approximates $v$. The separation estimate just
proved and their Proposition~6.1 bound the error by $CR^{-2}$
in $C^1$, including between the zero surfaces
\cite[equation~(11.1)]{WW19}. The centres of these functions
differ slightly from the actual zero surfaces. Their Lemma~4.6
bounds these displacements by $CR^{-2}$ in $C^1$.

We obtain $G_0$ from this sum by moving the centres back to the
zero surfaces and then changing the cutoff. The first change
alters each function and its first derivatives by at most
$CR^{-2}e^{-c|s|}$, because $q'$ and $q''$ decay exponentially.
Here \eqref{curvest} bounds the derivatives of the nearest-point
projection. The ordered separation of the surfaces makes the
sum of these errors at most $CR^{-2}$, independently of their
number. For the second change, both cutoffs act only where
the differences from the limiting constants and their first
derivatives are at most $CR^{-4}$. This follows from
\cite[Section~4.1]{WW19} for their cutoff and from
$|s|\geq4\log R$ for ours. The same summation bounds the
resulting $C^1$ error by $CR^{-2}$.

After these changes, the sum in \cite[equation~(4.3)]{WW19}
is precisely the function $G_0$ defined above.
Across each zero surface, $\sgn v$ and the corresponding
$\sgn s$ have the same jump, since $s$ is oriented towards
increasing $v$. Their difference is therefore smooth there,
and so is $G_0$. The preceding bounds prove \eqref{gluingest}
throughout the smaller graphical region.
Within the smaller region under consideration, points at
distance less than $cR$ from the zero set are covered by
graphical regions about nearest zeros.
At the remaining points, every truncated difference vanishes,
so $G_0=\sgn v$ and Lemma~\ref{clearing} proves the same estimate.
All domain restrictions use fixed factors, and all sums have
bounds independent of the number of surfaces.
\end{proof}

We then  have the following
\begin{corollary}
We fix $\delta_*$ as in Proposition~\ref{localest}.  If
$Z_*\ne\varnothing$, then, for every $0<\tau<1$,
\begin{equation}\label{rbound}
 r(X)\leq C_\tau d(X)^{-2}
 \qquad\text{when }|v(X)|\leq\tau
 \text{ and }d(X)\text{ is sufficiently large}.
\end{equation}
\end{corollary}

\begin{proof}
Lemma~\ref{clearing} bounds the distance of a point with
$|v(X)|\leq\tau$ from the zero set by a constant depending only on
$\tau$. We choose a nearest zero and a ball whose radius is a fixed
fraction of $d(X)$ and whose larger concentric ball avoids $Z_*$.
At the resulting bounded normal distance $s$,
Proposition~\ref{localest} gives
\[
 |v-q(s)|+|\nabla v-q'(s)\nabla s|\leq C_\tau d(X)^{-2}.
\]
We have $h\geq1-\tau^2>0$ at $X$, and the preceding estimate makes
$1-q(s)^2$ bounded below as well when $d(X)$ is large. The identity
$2q'(s)^2=(1-q(s)^2)^2$, together with $|\nabla s|=1$, therefore implies
\[
 |h^2-2|\nabla v|^2|\leq C_\tau d(X)^{-2}.
\]
Division by $h^2\geq(1-\tau^2)^2$ proves \eqref{rbound}.
\end{proof}

\section{Test functions near the zero set}
In this section, we use the stability of $v$ in $\R^3$
to derive an inequality for functions on the regular zero
surface.
The estimates in Section~\ref{prelims} show that, sufficiently
far from $Z_*$, the zero surfaces are smooth and separated,
and the solution near each surface is close to a
one-dimensional transition.
We use the derivative of this transition to extend a cutoff
function from the surface to a test function in $\R^3$.
When several such extensions overlap, we add their
contributions.

We estimate the second variation by separating the terms
associated with one surface from those involving two
distinct surfaces.
The calculation on a single surface produces a negative
term involving its squared curvature and a positive term
involving the squared gradient of the cutoff.
We control the terms involving distinct surfaces using
the separation estimates from the preceding section.
Stability then bounds the curvature integral by the gradient
integral and an error that decreases with the distance
from $Z_*$, with constants independent of the number
of surfaces.

We assume throughout this section that $Z_*\ne\varnothing$, and we use the
distance $d(X)=\dist(X,Z_*)$ from Section~\ref{prelims}. For a sufficiently
large number $a$, we consider
\[
\Gamma_a=\{y\in\Sigma:d(y)>a\},\qquad
\nu(y)=\frac{\nabla v(y)}{|\nabla v(y)|}.
\]
Thus $\Gamma_a$ is the smooth part of the zero set on which our construction
takes place, and $\nu$ points in the direction in which $v$ increases.
We use a smooth approximation $\widetilde d$ of the distance, with
\begin{equation}\label{smoothdist}
 c d\leq\widetilde d\leq C d,\qquad
 |D\widetilde d|\leq C,\qquad
 |D^2\widetilde d|\leq C/d.
\end{equation}
The regularized-distance construction in~\cite[Chapter VI]{Stein70}
provides these bounds. Its role here is to make the length of a cutoff depend
smoothly on its base point.

We set $T_y=4\log(a+\widetilde d(y))$. We choose an even smooth function
$\kappa$ which equals one on $[-1,1]$, equals zero outside $[-2,2]$, and
decreases on the positive half-axis. All its derivatives vanish at $-2$
and $2$. With $q(s)=\tanh(s/\sqrt2)$, we define
\begin{equation}\label{cutfun}
\begin{split}
 Q_y(s)&=\sgn s+\kappa(s/T_y)\bigl(q(s)-\sgn s\bigr),\\
 p_y(s)&=\sqrt2\,\partial_s Q_y(s).
\end{split}
\end{equation}
The function $Q_y$ agrees with $q$ near zero and reaches the values $-1$
and $1$ outside a finite interval. In particular, $Q_y$ is smooth at zero,
$p_y\geq0$, and $p_y$ is supported in $|s|\leq2T_y$.

An individual point $X\in\R^3$ may lie on perpendicular segments from several
different zero surfaces. The next lemma allows us to add their contributions
without choosing a numbering for all the surfaces in space. The number of
terms at any one point is bounded by a fixed constant. A compactly supported
coefficient on $\Gamma_a$ therefore produces a compactly supported test
function in $\R^3$.

\begin{lemma}\label{sumlemma}
For all sufficiently large $a$, the map
\[
 E(y,s)=y+s\nu(y),\qquad y\in\Gamma_a,
\]
has an invertible differential on $|s|\leq2T_y$. On $\{d>3a\}$ the functions
\[
 G(X)=\sgn v(X)+\sum_{E(y,s)=X}\bigl(Q_y(s)-\sgn s\bigr),
 \qquad P(X)=\sum_{E(y,s)=X}p_y(s)
\]
are smooth. The sums include the pairs with $|s|\leq2T_y$, and their number
is bounded by a fixed constant. For $\eta\in C_c^\infty(\Gamma_a)$, the formula
\begin{equation*}
 F_\eta(X)=\sum_{E(y,s)=X}\eta(y)p_y(s)
\end{equation*}
defines a smooth compactly supported function on $\R^3$.
\end{lemma}

\begin{proof}
We first locate all the possible base points $y$ of a fixed point $X$.
We then use the separation of the local zero surfaces to bound their number.
These two facts allow us to check smoothness by using finitely many local
inverse maps.

If $X=E(y,s)$ and $|s|\leq2T_y$, the distance function satisfies
\[
 |d(y)-d(X)|\leq8\log(a+C d(y)).
\]
We increase the fixed lower bound on $a$ so that
$8\log(a+Ct)\leq t/4$ for every $t\geq a$. Hence
$3d(y)/4\leq d(X)\leq5d(y)/4$, and in particular
\[
 \frac12d(X)\leq d(y)\leq2d(X),
 \qquad |y-X|\leq C\log d(X).
\]
When $d(X)>3a$, these base points stay a positive
distance from $d=a$. Proposition~\ref{localest}, applied on a ball whose
radius is a fixed small multiple of $d(X)$, places all these points in the
same collection of local graphs. Their curvatures satisfy
$|A|+|\nabla A|\leq C/d(X)$, and the distance between distinct relevant
graphs is at least $\sqrt2\log d(X)-C$.

It follows that $|sA|<1/2$. We denote a local inverse of $E$ by
$(\pi_\alpha,s_\alpha)$, where $\pi_\alpha(X)$ is its base point on the zero
surface. Differentiating $(I-sA)^{-1}$ yields
\begin{equation}\label{branches}
\begin{aligned}
 |D\pi_\alpha|&\leq C,&
 |D^2\pi_\alpha|&\leq \frac{C(1+\log d(X))}{d(X)},\\
 |Ds_\alpha|&=1,& |D^2s_\alpha|&\leq C/d(X).
\end{aligned}
\end{equation}
Only the derivative estimate $|\nabla A|\leq C/d(X)$ is used here.

We also need uniqueness of a base point on any one local graph. We express
that graph as $(z,f(z))$. On the line segment between two proposed base
points, the larger graphical domain from Proposition~\ref{localest}
ensures $|f-X_3|\leq C\log d(X)$ and $|D^2f|\leq C/d(X)$. We have
\[
 \frac12 D_z^2\bigl(|z-X'|^2+|f(z)-X_3|^2\bigr)
 =I+Df\otimes Df+(f-X_3)D^2f\geq\frac12 I.
\]
Two base points would be two critical points of this strictly convex
function. Hence each graph supplies at most one. The remaining base points
are separated by $c\log d(X)$ and lie in $B_{C\log d(X)}(X)$.
Disjoint balls of radius $c\log d(X)/3$ around them bound their number by a
fixed constant. This argument also counts the points with $|s|=2T_y$.

Over a compact subset of $\{d>3a\}$, all relevant pairs $(y,s)$ form a
compact set away from $d=a$. Locally, the displayed sums are therefore finite
sums of smooth functions obtained from the inverse maps. When a term enters
or leaves the sum at $|s|=2T_y$, all its derivatives vanish. At a zero of
$v$, the jump of its own $-\sgn s$ cancels the jump of $\sgn v$.
Thus $G$ and $P$ extend smoothly across the zero set.

For $F_\eta$, the function $T_y$ is bounded on the compact set
$\supp\eta$. Thus the pairs satisfying $y\in\supp\eta$ and
$|s|\leq2T_y$ form a compact set, and their image under the continuous
map $E$ is compact as well. This image contains $\supp F_\eta$.
The coefficient $\eta$ vanishes near the boundary of
$\Gamma_a$, and the preceding argument proves smoothness throughout
$\R^3$. The support lies in $\{d>a/2\}$. In the part where $d\leq3a$,
it is the factor $\eta$, rather than the unrestricted sum $P$, that
ensures smoothness at the boundary of the chosen part of the zero set.
\end{proof}

The function $G$ approximates $v$, while the positive sum $P$ approximates
$h=1-v^2$. We need both the values and the first derivatives of these
approximations when we join the inner and outer test functions later.
The estimate below follows from the local expansion and the separation
between the zero surfaces. We include the calculation for the sum $P$ because
it also explains the signs used in the second variation.

\begin{lemma}\label{cutmatch}
On $\{d>3a\}$, we have
\begin{equation*}
 |v-G|+|\nabla(v-G)|+|P-h|+|\nabla(P-h)|\leq C d^{-2}.
\end{equation*}
\end{lemma}

\begin{proof}
We fix $X$ and compare $G$ with the sum in \eqref{gluingest} on a ball of
radius $c d(X)$. The cutoff in that estimate starts at
$4\log(c d(X))$, whereas the present cutoffs start at
$T_y=4\log d(X)+O(1)$. Their difference is supported where $q-\sgn s$
and its first derivatives have size at most $C d(X)^{-4\sqrt2}$.
The inverse-map bounds \eqref{branches} and
$|D(T_y\circ\pi_\alpha)|\leq C/d(X)$ control differentiation of the
cutoffs. The bounded number of terms makes the total $C^1$ difference at
most $C d(X)^{-3}$. This proves the estimates for $v-G$.

The same comparison applies when the sum is empty. Indeed, if $X$ is at
distance at least $c d(X)$ from the entire zero set, Lemma~\ref{clearing}
applies. Otherwise, a nearest zero has distance comparable to $d(X)$ from
$Z_*$. The larger ball around that zero contains both $X$ and all the
segments under consideration, so \eqref{gluingest} applies there.

We next compare $P$ with $1-G^2$. On a region disjoint from the zero set,
we use the local expressions $Q_\alpha$ and $p_\alpha$ supplied by
$(\pi_\alpha,s_\alpha)$. We have
\begin{equation}\label{tailcalc}
\begin{split}
 |Q_\alpha-\sgn s_\alpha|
 +|\nabla(Q_\alpha-\sgn s_\alpha)|
 &\leq C e^{-\sqrt2|s_\alpha|},\\
 |p_\alpha-(1-Q_\alpha^2)|
 +|\nabla(p_\alpha-(1-Q_\alpha^2))|
 &\leq C d(X)^{-5}.
\end{split}
\end{equation}
For the second line, the scalar identity behind the estimate is
\[
 p-(1-Q^2)=\kappa(\kappa-1)(q-\sgn s)^2
       +\sqrt2\,\kappa_s(q-\sgn s).
\]
The difference is confined to the cutoff interval, where the exponential
decay gives the asserted bound, also after one derivative.

For distinct local representations we have
$|s_\alpha|+|s_\beta|\geq\sqrt2\log d(X)-C$.
If $\sgn s_\alpha\ne\sgn v(X)$, the segment from
$\pi_\alpha(X)$ to $X$ crosses another zero. This second zero lies on a
different local graph, since a point on the same graph would contradict
the strict minimum of the distance function proved in Lemma~\ref{sumlemma}.
Consequently $|s_\alpha|\geq\sqrt2\log d(X)-C$ in this case.

We now expand the squares directly:
\[
\begin{split}
 (1-G^2)-\sum_\alpha(1-Q_\alpha^2)
 ={}&2\sum_\alpha(\sgn s_\alpha-\sgn v)
                        (Q_\alpha-\sgn s_\alpha)\\
 &-2\sum_{\alpha<\beta}(Q_\alpha-\sgn s_\alpha)
                            (Q_\beta-\sgn s_\beta).
\end{split}
\]
The linear terms with matching signs vanish. The other linear terms and all
products are bounded in $C^1$ by $C d(X)^{-2}$, by the separation just
proved. We combine this identity with \eqref{tailcalc} and the estimate
for $v-G$ to obtain the two bounds for $P-h$. Smoothness extends them to
the zero set.
\end{proof}

The next estimates control products coming from two different zero surfaces.
The separation of the surfaces makes their product small, even when both
functions are nonzero at the same point. We will integrate a mixed term
first with one surface held fixed and then with the other held fixed.
The change between these two integrations is simply the area formula for
the map $E$.

\begin{lemma}\label{branchlem}
On every local representation of $E$, we have
\begin{equation}\label{decay}
 |p|+|p_s|+|\nabla_y p|+|\nabla_Xp_\alpha|
 \leq C e^{-\sqrt2|s|}.
\end{equation}
For every fixed nonnegative integer $k$, we also have
\begin{equation}\label{overlap}
 \int (1+|s_\alpha|)^k
       \sum_{\beta\ne\alpha}e^{-\sqrt2(|s_\alpha|+|s_\beta|)}
       J_\alpha\,ds_\alpha
 \leq C_k d(y_\alpha)^{-2}(\log d(y_\alpha))^{k+1}.
\end{equation}
Here $J_\alpha$ is the Jacobian of $E$ on the corresponding local surface.
The same estimate holds with the roles of $\alpha$ and $\beta$ exchanged
after integration over $\R^3$. Finally, the map $\eta\mapsto F_\eta$
extends continuously from $H_c^1(\Gamma_a)$ to $H_c^1(\R^3)$ and satisfies
\begin{equation}\label{ampgrad}
 \|F_\eta\|_{H^1(\R^3)}^2
 \leq C\int_{\Gamma_a}(|\eta|^2+|\nabla_\Sigma\eta|^2)\,dA.
\end{equation}
\end{lemma}

\begin{proof}
The definition of $p$ on the cutoff interval reads
\[
 p=\sqrt2\left[\kappa(s/T_y)q'(s)
       +T_y^{-1}\kappa'(s/T_y)(q(s)-\sgn s)\right].
\]
Differentiation in $T_y$ introduces factors such as $s/T_y^2$ and
$T_y^{-2}$. We have $|\nabla T_y|\leq C/d(y)$, while for $|s|<T_y$
the function $p=\sqrt2q'$ is independent of $y$. Together with
\eqref{branches}, these facts prove \eqref{decay}.

At an intersection of two perpendicular segments, the separation estimate
implies
\[
 e^{-\sqrt2(|s_\alpha|+|s_\beta|)}\leq C d(X)^{-2},
 \qquad
 \frac12d(X)\leq d(y_\gamma)\leq2d(X)
 \quad(\gamma\in\{\alpha,\beta\}).
\]
The integration interval has length at most $C\log d(y_\alpha)$,
the Jacobian satisfies $1/4\leq J_\alpha\leq9/4$, and the number of terms
is bounded. Moreover, $1+|s_\alpha|\leq C\log d(y_\alpha)$ on this
interval. Multiplication of these bounds gives
$C_kd(y_\alpha)^{-2}(\log d(y_\alpha))^{k+1}$, which proves
\eqref{overlap}.

To explain the exchange of the two integrations, we take any nonnegative
function $\Psi$ of an ordered pair of local representations. We extend it
by zero where either of the two functions is zero. The area formula and
Tonelli's theorem yield
\begin{equation}\label{fibint}
\begin{split}
 \int_{\Gamma_a}\int J_\alpha
             \sum_{\beta\ne\alpha}\Psi(\alpha,\beta)\,ds_\alpha\,dA_\alpha
 &=\int_{\R^3}\sum_{\alpha\ne\beta}\Psi(\alpha,\beta)\,dX\\
 &=\int_{\Gamma_a}\int J_\beta
             \sum_{\alpha\ne\beta}\Psi(\alpha,\beta)\,ds_\beta\,dA_\beta.
\end{split}
\end{equation}
At a fixed $X$, the indices $\alpha$ and $\beta$ label local inverse
maps of $E$. The integral over $\R^3$ counts each ordered pair of
distinct representations exactly once. The first and last integrals
count these same pairs, starting respectively from the first and the
second point on the zero surface. The exchange is therefore
valid even when several perpendicular segments reach $X$.
The function $\Psi$ can contain a weight or the indicator of a region in
$\R^3$. A factor involving $s_\alpha$ remains unchanged during this
exchange. In the last integral its absolute value is bounded by
$C\log d(y_\beta)$, which proves the asserted version of \eqref{overlap}.

For the $H^1$ bound, the bounded number of terms and the chain rule give
\[
\begin{split}
 |F_\eta|^2+|\nabla F_\eta|^2
 \leq C\sum_\alpha e^{-2\sqrt2|s_\alpha|}
 \bigl(&|\eta(\pi_\alpha(X))|^2\\
       &+|\nabla_\Sigma\eta(\pi_\alpha(X))|^2\bigr).
\end{split}
\]
We have used $|D\pi_\alpha|\leq C$ and \eqref{decay} here.
The area formula changes the integral of each summand into an integral
over $y$ and $s$. The bounds for $J$ and the finite integral
$\int_\R e^{-2\sqrt2|s|}\,ds$ prove \eqref{ampgrad}.
We approximate an $H^1$ coefficient on a fixed compact
subset of $\Gamma_a$ to obtain the extension. Coefficients between zero
and one admit approximations with the same bounds. When a coefficient
equals zero or one on specified open regions, we fix smaller closed regions
inside them and use a smooth cutoff to preserve those exact values during
approximation. All the resulting functions $F_\eta$ have support in one
fixed compact subset of $\R^3$, and their quadratic forms converge.
\end{proof}

We next estimate the second variation of $F_\eta$. Wang and Wei use
derivatives of one-dimensional transitions in stability tests
in~\cite[Section 8]{WW19}. Here the derivatives are based on the actual
zero surfaces, and we need bounds for their sum which are independent of
the total number of these surfaces.

We use the convention $A=-D\nu$, and we denote the mean curvature by
$H_0=\operatorname{tr}A$ and the Gaussian curvature by $K_\Sigma=\det A$.
In the coordinates $E(y,s)$, the Euclidean volume and metric satisfy
\begin{equation}\label{normalgeo}
\begin{aligned}
 J&=\det(I-sA)=1-H_0s+K_\Sigma s^2,
 &g_s&=(I-sA)^2g_0,\\
 \Delta_E&=\partial_s^2-H(s)\partial_s+\Delta_s,
 &H(s)J&=H_0-2K_\Sigma s.
\end{aligned}
\end{equation}
Here $\Delta_s$ differentiates along the surface at distance $s$, and
$\Delta_E$ is the Euclidean Laplacian expressed in these coordinates.
On $|s|\leq2T_y$ we have $|H(s)|\leq C d(y)^{-2}(1+|s|)$.
The fixed integral of the squared one-dimensional derivative is
\[
 \sigma=\int_\R(\sqrt2q'(s))^2\,ds=\frac{4\sqrt2}{3}.
\]

We first examine the contribution from one zero surface. The comparison
between $v$ and the one-dimensional transition makes the potential term
nonpositive. Integration in the perpendicular direction then leaves a
negative multiple of $|A|^2$. This is the curvature term that will control
cutoff functions on the zero surface in the next section.

\begin{lemma}\label{diag}
For each $y\in\Gamma_a$, we have
\[
 \int p_y L_vp_y J\,ds
 \leq-\frac\sigma2|A(y)|^2+C d(y)^{-4}.
\]
\end{lemma}

In this formula, $L_vp_y$ means that $L_v$ acts on the function defined
in $\R^3$ by the local inverse coordinates $(y,s)$.

\begin{proof}
We first use Modica's estimate to determine the sign of the potential
term. We then integrate the mean-curvature term exactly, and bound the
remaining terms where the cutoff varies.

Since $\theta(y)=0$ and $|\nabla\theta|\leq1$, the segment from $y$ to
$E(y,s)$ satisfies
\[
 |v(E(y,s))|\leq|q(s)|\leq|Q_y(s)|.
\]
This argument also applies when the segment meets another zero surface.
Differentiating \eqref{cutfun} and using \eqref{normalgeo}, we obtain
\begin{equation}\label{normalop}
\begin{split}
 L_vp&=3\bigl(v(E)^2-Q^2\bigr)p+H(s)p_s+\mathcal R,\\
 \mathcal R&=-\Delta_s p
       -\sqrt2\,\partial_s\bigl(Q_{ss}-Q^3+Q\bigr).
\end{split}
\end{equation}
For $|s|<T_y$, the function $p=\sqrt2q'$ is independent of $y$ and
$\mathcal R=0$. In the remaining interval, we have
\[
 |\nabla_sT_y|\leq C/d(y),\qquad
 |\Delta_sT_y|\leq C d(y)^{-2}(1+|s|).
\]
Indeed, for a function $f$ depending only on $y$, the formula for
$\Delta_s$ is
\[
 \Delta_s f=J^{-1}\operatorname{div}_{g_0}
       \bigl(J(I-sA)^{-2}\nabla_{g_0}f\bigr).
\]
Equation \eqref{smoothdist} bounds $|\nabla_{g_0}T_y|$ by $C/d$
and $|\nabla_{g_0}^2T_y|$ by $C/d^2$. The derivatives of the coefficients
are controlled by $|A|+|\nabla A|\leq C/d$. Thus the weaker estimate
$|\nabla A|\leq C/d$ is sufficient for this computation.

The derivatives of $p$ with respect to $T_y$, up to order two, and the
last derivative in \eqref{normalop} are bounded by $C e^{-\sqrt2|s|}$
where the cutoff varies. Since $|s|\geq4\log(a+\widetilde d(y))$ there,
we obtain $|\mathcal R|\leq C d(y)^{-5}$.

The first term of \eqref{normalop} is nonpositive. All derivatives of $p$
vanish at $s=\pm2T_y$, so integration by parts yields
\[
 \int H(s)pp_sJ\,ds
 =\int(H_0-2K_\Sigma s)pp_s\,ds
 =K_\Sigma\int p^2\,ds.
\]
We have $\int p^2\,ds=\sigma+O(d^{-10})$, and the term containing
$\mathcal R$ is bounded by $C d^{-4}$. Finally,
$2K_\Sigma=H_0^2-|A|^2$ and $|H_0|\leq C d^{-2}$ imply the result.
\end{proof}

We now separate the terms from a single surface from those involving two
different surfaces. For a smooth nonnegative coefficient $\eta$, we use
$F=F_\eta$ and $\eta_\alpha=\eta\circ\pi_\alpha$ locally. All gradients
in the following formulas are taken in $\R^3$. We denote by $V_p$ the
part of $\nabla F$ in which the derivative falls on $p$, so that
\[
 V_p=\sum_\alpha\eta_\alpha\nabla p_\alpha,
 \qquad
 \nabla F=\sum_\alpha p_\alpha\nabla\eta_\alpha+V_p.
\]
We will use the following two expressions several times. The first includes
all the contributions, while the second contains just those from one
surface at a time:
\begin{equation}\label{products}
\begin{split}
 I_\eta={}&F\sum_\beta\eta_\beta L_vp_\beta
       +\left|\sum_\alpha p_\alpha\nabla\eta_\alpha\right|^2\\
 &+\sum_{\alpha,\beta}\eta_\alpha
       (p_\beta\nabla p_\alpha-p_\alpha\nabla p_\beta)
                       \cdot\nabla\eta_\beta,\\
 D_\eta={}&\sum_\alpha
       \bigl(p_\alpha^2|\nabla\eta_\alpha|^2
            +\eta_\alpha^2p_\alpha L_vp_\alpha\bigr).
\end{split}
\end{equation}
We have
\[
 \operatorname{div}V_p
 =\sum_\alpha\nabla\eta_\alpha\cdot\nabla p_\alpha
  +\sum_\alpha\eta_\alpha\Delta p_\alpha.
\]
To verify the formula below, we expand
$\operatorname{div}(FV_p)=\nabla F\cdot V_p+F\operatorname{div}V_p$.
The terms containing $\Delta p_\alpha$, together with
$(3v^2-1)F^2$, become $F\sum_\alpha\eta_\alpha L_vp_\alpha$.
The remaining mixed derivative terms are
\[
\begin{split}
 &\left(\sum_\beta p_\beta\nabla\eta_\beta\right)\cdot V_p
       -F\sum_\beta\nabla\eta_\beta\cdot\nabla p_\beta\\
 &\qquad=\sum_{\alpha,\beta}\eta_\alpha
     (p_\beta\nabla p_\alpha-p_\alpha\nabla p_\beta)
         \cdot\nabla\eta_\beta.
\end{split}
\]
The summand vanishes when $\alpha=\beta$. Together with the square
of $\sum_\alpha p_\alpha\nabla\eta_\alpha$, these are exactly the
terms in $I_\eta$. We therefore have obtained the identity
\begin{equation}\label{product}
 |\nabla F|^2+(3v^2-1)F^2
 =I_\eta+\operatorname{div}(FV_p).
\end{equation}
In particular, the terms involving distinct surfaces occur explicitly in
$I_\eta$. We will also use the same identity with a weight, which explains
the need to keep the divergence term in this form.

The difference $I_\eta-D_\eta$ consists of terms involving
two distinct zero surfaces.
The main difficulty comes from products of the form
$|\eta_\alpha|\,|\nabla\eta_\beta|$, where the value of $\eta$
and its derivative are evaluated using projections onto
different surfaces.
We first apply Young's inequality to separate this product
into terms containing $\eta_\alpha^2$ and
$|\nabla\eta_\beta|^2$.
For the first term, we fix the point on $\Sigma_\alpha$
and integrate along its perpendicular segment, whereas
for the second term we fix the point on $\Sigma_\beta$
and integrate along that surface's perpendicular segment.
The separation estimate contributes a factor $d^{-2}\log d$
in each case, allowing us to make the coefficient of the
gradient integral small while bounding the remaining
contribution by an integral involving
$d^{-4}(\log d)^2\eta^2$.

\begin{lemma}\label{mixed}
For a continuous function $0\leq w\leq1$ on $\R^3$, we set
\[
 \Theta_w(y)=\sup_{|s|\leq2T_y}w(E(y,s)),
 \qquad m(d)=d^{-4}(\log d)^2.
\]
For all sufficiently large $a$, we have
\begin{equation}\label{mixedest}
\begin{split}
 \int_{\R^3}w(I_\eta-D_\eta)
 \leq{}&\int_{\Gamma_a}\Theta_w
       \left(\frac\sigma4+C d^{-2}\log d\right)
                         |\nabla_\Sigma\eta|^2\,dA\\
 &+C\int_{\Gamma_a}\Theta_w m(d)\eta^2\,dA.
\end{split}
\end{equation}
The constants are independent of $\supp\eta$ and of the total number of
zero surfaces.
\end{lemma}

\begin{proof}
We organize the terms according to whether they involve the potential,
the mean curvature, or derivatives of $\eta$. We use the sign of the
potential term before estimating the absolute values of the remaining
terms. The exchange formula \eqref{fibint} then allows each estimate to
be expressed as an integral over the surface carrying the relevant
coefficient $\eta$ or its derivative.

We substitute \eqref{normalop} in the terms with $\alpha\ne\beta$.
Their potential contribution is
\[
 3w\sum_\beta\eta_\beta p_\beta(v^2-Q_\beta^2)
                  (F-\eta_\beta p_\beta)\leq0.
\]
Every factor outside the parentheses $v^2-Q_\beta^2$ is nonnegative,
and the comparison in Lemma~\ref{diag} determines the sign of those
parentheses. We therefore discard this entire contribution.

For the mean-curvature terms, we have
\[
 |H_\beta p_{\beta,s}|
 \leq C d(X)^{-2}(1+|s_\beta|)e^{-\sqrt2|s_\beta|}.
\]
We use $2\eta_\alpha\eta_\beta\leq\eta_\alpha^2+\eta_\beta^2$,
exchange the integrations by \eqref{fibint}, and apply \eqref{overlap}
with $k=1$. The resulting bound is
$C\int\Theta_w m(d)\eta^2$. The remaining term $\mathcal R_\beta$
has size at most $C d(X)^{-5}$ and is supported where its cutoff varies.
For the integral based at $y_\alpha$, we use
$\int p_\alpha J_\alpha\,ds_\alpha\leq C$.
After exchanging the two integrations, we instead use
$\sum_\alpha p_\alpha\leq C$ and an interval of length at most
$C\log d(y_\beta)$. Both contributions are bounded by
$C\int\Theta_w d^{-5}\log d\,\eta^2$, which is absorbed by the
preceding bound.

The products of two derivatives of $\eta$ have absolute value at most
\[
 C e^{-\sqrt2(|s_\alpha|+|s_\beta|)}
                 |\nabla\eta_\alpha|\,|\nabla\eta_\beta|.
\]
The inequality $2ab\leq a^2+b^2$, followed by \eqref{fibint} and
\eqref{overlap} with $k=0$, bounds their integral by
$C\int\Theta_w d^{-2}\log d\,|\nabla_\Sigma\eta|^2$.

It remains to estimate the mixed first-derivative terms in \eqref{products}.
Equation \eqref{decay} implies
\[
 |p_\beta\nabla p_\alpha-p_\alpha\nabla p_\beta|
 \leq C e^{-\sqrt2(|s_\alpha|+|s_\beta|)}.
\]
We choose a fixed large number $M$. At the point $X$, Young's inequality
takes the form
\[
 |\eta_\alpha|\,|\nabla\eta_\beta|
 \leq \frac{M\log d(X)}{2d(X)^2}\eta_\alpha^2
       +\frac{d(X)^2}{2M\log d(X)}|\nabla\eta_\beta|^2.
\]
For the first term we integrate with $y_\alpha$ fixed. The estimate is
\[
 \int J_\alpha\sum_{\beta\ne\alpha}w(X)
       \frac{\log d(X)}{d(X)^2}
       e^{-\sqrt2(|s_\alpha|+|s_\beta|)}\,ds_\alpha
 \leq C\Theta_w(y_\alpha)
             d(y_\alpha)^{-4}(\log d(y_\alpha))^2.
\]
For the second term, we first exchange the two integrations in
\eqref{fibint}. We can then use $|D\pi_\beta|\leq C$ with $y_\beta$
fixed, obtaining
\[
\begin{split}
 &\int J_\beta\sum_{\alpha\ne\beta}w(X)
       \frac{d(X)^2}{\log d(X)}
       e^{-\sqrt2(|s_\alpha|+|s_\beta|)}
                    |\nabla_X\eta_\beta|^2\,ds_\beta\\
 &\hspace{35mm}\leq C\Theta_w(y_\beta)
                         |\nabla_\Sigma\eta(y_\beta)|^2.
\end{split}
\]
The two choices of integration order explain the coefficients explicitly.
For the term containing $\eta^2$, the overlap estimate contributes
$d^{-2}\log d$, which multiplies $M d^{-2}\log d$ from Young's
inequality. For the term containing $|\nabla_\Sigma\eta|^2$, we first
fix the point carrying this derivative by exchanging the integrals.
The same overlap bound then multiplies $d^2/(M\log d)$. Thus
\[
 \frac{M\log d}{d^2}\frac{\log d}{d^2}
       =M d^{-4}(\log d)^2,
 \qquad
 \frac{d^2}{M\log d}\frac{\log d}{d^2}=\frac1M.
\]
The mixed terms consequently contribute at most
\[
 CM\int\Theta_w m(d)\eta^2
       +\frac{C}{M}\int\Theta_w|\nabla_\Sigma\eta|^2.
\]
We fix $M$ so that $C/M\leq\sigma/4$. Its value is then included in
the constant multiplying the integral of $m(d)\eta^2$. Combining the bounds proves \eqref{mixedest}.
\end{proof}

 We now apply the preceding estimates to the test function
$F_\eta$, constructed from a nonnegative function $\eta$
on the regular zero surface $\Gamma_a$.
The upper bound for $Q_v(F_\eta)$ contains a negative term
involving $|A|^2\eta^2$, together with the integral of
$|\nabla_\Sigma\eta|^2$ and an error that decreases with
the distance from $Z_*$.
Since stability gives $Q_v(F_\eta)\geq0$, the absolute value
of the negative curvature term is bounded by the other
two contributions.
The following proposition states the resulting inequality
on $\Gamma_a$, with constants independent of the number
of zero surfaces.

\begin{proposition}\label{surface}
For a sufficiently large fixed lower bound on $a$, every nonnegative
$\eta\in H_c^1(\Gamma_a)$ satisfies
\begin{equation}\label{quadbound}
 Q_v(F_\eta)\leq\sigma\int_{\Gamma_a}
       \left(\frac32|\nabla_\Sigma\eta|^2
                       -\frac12|A|^2\eta^2\right)\,dA
       +C\int_{\Gamma_a}m(d)\eta^2\,dA.
\end{equation}
Consequently, every real-valued $\eta\in H_c^1(\Gamma_a)$ satisfies
\begin{equation}\label{surfstab}
 \frac12\int_{\Gamma_a}|A|^2\eta^2\,dA
 \leq\frac32\int_{\Gamma_a}|\nabla_\Sigma\eta|^2\,dA
       +C\int_{\Gamma_a}m(d)\eta^2\,dA.
\end{equation}
\end{proposition}

\begin{proof}
We integrate \eqref{product} and estimate separately the terms in
$D_\eta$ and $I_\eta-D_\eta$. The area formula expresses the first
integral as
\[
 \int_{\Gamma_a}\left[
       \int p^2Jg_s^{-1}(\nabla\eta,\nabla\eta)\,ds
        +\eta^2\int pL_vpJ\,ds\right]dA.
\]
We have $|Jg_s^{-1}-I|\leq C d^{-1}|s|$. The bound
$\int |s|p(s)^2\,ds\leq C$ therefore yields
\[
 \int p^2Jg_s^{-1}\,ds=\sigma I+O(d^{-1}).
\]
Lemma~\ref{diag} bounds the second integral, and Lemma~\ref{mixed}
with $w=1$ bounds $I_\eta-D_\eta$. The vector $FV_p$ is compactly
supported, so its divergence integrates to zero. We increase the lower
bound on $a$ until $C d^{-1}+C d^{-2}\log d\leq\sigma/4$.
Together with the $\sigma/4$ already allowed in Lemma~\ref{mixed},
this proves \eqref{quadbound} for smooth coefficients.

Lemma~\ref{branchlem} extends the estimate to $H^1$ coefficients.
Stability gives \eqref{surfstab} for nonnegative $\eta$. Application
to $|\eta|$ proves the same inequality for every real-valued $\eta$.
\end{proof}

\section{Cutoff functions on the outer part of the zero set}

In this section, we construct cutoff functions on the regular
zero surface outside $K$ and estimate the integrals of their
squared gradients.
We need these integrals to be bounded by a small constant
times $N$, so that the resulting errors can be compared
with the negative contribution near $K$ from
Proposition~\ref{clusters}.
We obtain these estimates by first bounding the area of
balls defined using distance along the surface, and then
constructing logarithmic cutoff functions.

The use of distance along the surface and the Gauss--Bonnet
formula also appears in the free boundary argument of
\cite[Section~10]{CFFS25}.
In our setting, the surface inequality \eqref{surfstab}
contains the additional term involving $m(d)$, and the
surface may have several inner boundary curves.
We attach a smooth auxiliary surface along these curves
and estimate the additional terms needed to extend the
quadratic inequality to the enlarged surface.
We can then apply the inequality of B\'erard and Castillon
\cite{BC14} to functions of the distance from a single point.
This yields a quadratic area bound, with its coefficient
proportional to $N$, from which the required cutoff
estimates follow.

Let us  fix a large matching scale $R$ and set $a=R/1000$, with $a$ above all earlier lower thresholds. The separation scale satisfies $L\geq100R$. Proposition~\ref{clusters} supplies a compact set $K$ and its number $N$ of disjoint unit balls. We set $d_K=\dist(\,\cdot\,,K)$ and smooth this distance by convolution at a fixed unit scale. The resulting function $\tau$ satisfies
\[
 |\tau-d_K|\leq C_0,
 \qquad |\nabla\tau|\leq1,
 \qquad |D^2\tau|\leq C_0.
\]
The identity $d=d_K$ holds throughout $\{d_K<L\}$.

The first estimate compares the area of the regular zero surface with the integral of $h$ in $\R^3$. A tube made of short segments perpendicular to the surface makes this comparison possible using only the uniform local estimates. The resulting bounds control the matching region and the integral of the potential $m(d)=d^{-4}(\log d)^2$.

\begin{lemma}
For every centre $x\in\R^3$ and every $s\geq e$,
\begin{equation}\label{areaest}
 \Area(\Gamma_a\cap B_s(x))\leq \frac{Cs^3}{\log s}.
\end{equation}
The region $A_R=\{R/2<\tau<3R\}$ satisfies
\begin{equation}\label{regmass}
 |A_R|\leq CNR^3,
 \qquad
 \int_{A_R}h+\Area(\Sigma\cap A_R)
 \leq \frac{CNR^3}{\log R}.
\end{equation}
Moreover,
\begin{equation}\label{potmass}
 \int_{\Gamma_a\cap\{R/100<d_K<L\}}m(d)\,dA
 \leq \frac{CN\log R}{R}.
\end{equation}
\end{lemma}
\begin{proof}
If \(y\in\Gamma_a\), then  \(v(y)=0\) and \(r(y)<\delta_*<1/2\). Therefore
\[
  |\nabla v(y)|^2
    =\frac{1}{2}|\nabla\theta(y)|^2
    =\frac{1-r(y)}{2}
    \geq\frac14.
\]
In particular, \(\Gamma_a\) is a smooth embedded surface. Set
\[
  \nu(y)=\frac{\nabla v(y)}{|\nabla v(y)|},
  \qquad
  E(y,t)=y+t\nu(y).
\]
By elliptic estimates, we have
\[ \|\nabla v\|_{L^\infty(\R^3)}
    +\|D^2v\|_{L^\infty(\R^3)}\leq M,\]
    where \(M\geq1\) is a fixed positive constant. Choose \(\varepsilon>0\) so small that $M\varepsilon\leq 1/8$.
Let $E(y, \nu)$ be the function defined in Lemma \ref{sumlemma}, we first verify that
\[
  E:\Gamma_a\times(-\varepsilon,\varepsilon)\longrightarrow\R^3
\]
is injective.

Suppose
\[
  y+t_1\nu(y)=z+t_2\nu(z),
  \qquad y,z\in\Gamma_a,\quad |t_1|,|t_2|<\varepsilon,
\]
and put \(\xi=z-y\). Taylor's theorem, the bound on \(D^2v\), and
\(v(y)=v(z)=0\) give
\[
  |\nabla v(y)\cdot\xi|\leq\frac{M}{2}|\xi|^2,
  \qquad
  |\nabla v(z)\cdot\xi|\leq\frac{M}{2}|\xi|^2.
\]
Dividing by the gradient lower bound yields
\[
  |\nu(y)\cdot\xi|\leq M|\xi|^2,
  \qquad
  |\nu(z)\cdot\xi|\leq M|\xi|^2.
\]
On the other hand, the assumed intersection implies
\(\xi=t_1\nu(y)-t_2\nu(z)\). Taking its scalar product with \(\xi\), we obtain
\begin{align*}
  |\xi|^2=t_1\nu(y)\cdot\xi-t_2\nu(z)\cdot\xi\leq M(|t_1|+|t_2|)|\xi|^2\leq\frac14|\xi|^2.
\end{align*}
Consequently \(\xi=0\), so \(y=z\) and then \(t_1=t_2\).
This proves injectivity for the entire short normal tube.

Fix \(x\in\R^3\) and \(s\geq e\), and write
\(D=\Gamma_a\cap B_s(x)\). Since
\[
  E\bigl(D\times(-\varepsilon,\varepsilon)\bigr)
    \subset B_{s+\varepsilon}(x),
\]
injectivity and the area formula give
\begin{align*}
  \int_{B_{s+\varepsilon}(x)}h\,dX
    &\geq
      \int_D\int_{-\varepsilon}^{\varepsilon}
         h(E(y,t))J(y,t)\,dt\,dA(y),
\end{align*}
where $J(y,t)$ is the volume Jacobian of the normal map. Our choice of \(\varepsilon\) implies
\[
  \frac12\leq J(y,t)\leq2
  \qquad \text{on}~\{(y, t):y\in\Gamma_a,\ |t|<\varepsilon\}.
\]
Also, using \(v(y)=0\) and the global gradient bound,
\[
  |v(E(y,t))|\leq M|t|\leq\frac12,
  \qquad
  h(E(y,t))\geq\frac34.
\]
Thus
\begin{align*}
  \int_{B_{s+\varepsilon}(x)}h\,dX\geq\frac{3\varepsilon}{4}\,\Area(D).
\end{align*}
The ambient integral estimate now yields
\[
  \Area\bigl(\Gamma_a\cap B_s(x)\bigr)
    \leq C\frac{(s+\varepsilon)^3}{\log(s+\varepsilon)}
    \leq C\frac{s^3}{\log s}.
\]
In the last inequality, we have used \(s\geq e\) and \(\varepsilon\leq1\).
This proves \eqref{areaest}.

For every \(X\in A_R\), the function \(\tau\) gives
\[
  \frac{R}{2}-C_0<d_K(X)<3R+C_0.
\]
Because \(R\geq6C_0\), it follows that
\[
  \frac{R}{3}<d_K(X)<4R<L.
\]
Thus \(d(X)=d_K(X)\) on \(A_R\). In particular, $\Sigma\cap A_R\subset\Gamma_a.$

Fix a point \(X\in A_R\). Because \(K\) is compact, we choose a nearest point \(z\in K\).
 The
covering of \(K\) gives \(|z-z_i|<2\)  for some \(i\), and hence
\[
  |X-z_i|\leq|X-z|+|z-z_i|
    <4R+2\leq5R.
\]
Consequently,
\[
  A_R\subset\bigcup_{i=1}^N B_{5R}(z_i).
\]
Summing the volumes of these balls proves
\[
  |A_R|\leq CNR^3.
\]
The ambient integral estimate gives
\[
  \int_{A_R}h\,dX
    \leq\sum_{i=1}^N\int_{B_{5R}(z_i)}h\,dX
    \leq C\frac{NR^3}{\log R}.
\]
Similarly,
\[
  \Area(\Sigma\cap A_R)
    \leq\sum_{i=1}^N
       \Area\bigl(\Gamma_a\cap B_{5R}(z_i)\bigr)
    \leq C\frac{NR^3}{\log R}.
\]
These inequalities establish \eqref{regmass}.

In order to prove \eqref{potmass}, we set
\[
  s_j=2^jR/100,\qquad
  \Omega_j=
    \Gamma_a\cap\{s_j<d_K\leq2s_j\}\cap\{d_K<L\},
  \qquad j=0,1,2,\ldots.
\]
Then
\[
  \Gamma_a\cap\{R/100<d_K<L\}
    =\bigcup_{j=0}^{\infty}\Omega_j.
\]

If \(y\in\Omega_j\), choose a nearest point \(z\in K\) and then an
index \(i\) with \(|z-z_i|<2\). Since \(s_j\geq R/100\geq e>2\),
\[
  |y-z_i|<2s_j+2\leq3s_j.
\]
It follows that
\begin{align*}
  \Area(\Omega_j)
    &\leq\sum_{i=1}^N
       \Area\bigl(\Gamma_a\cap B_{3s_j}(z_i)\bigr)\\
    &\leq C N\frac{s_j^3}{\log s_j}.
\end{align*}
On \(\Omega_j\) we have \(d=d_K<L\), so
\[
  m(d)
    =d_K^{-4}(\log d_K)^2
    \leq s_j^{-4}\bigl(\log(2s_j)\bigr)^2
    \leq C s_j^{-4}(\log s_j)^2.
\]
Combining the last two estimates yields
\[
  \int_{\Omega_j}m(d)\,dA
    \leq C N\frac{\log s_j}{s_j}.
\]
We may therefore sum over all \(j\geq0\):
\begin{align*}
  \int_{\Gamma_a\cap\{R/100<d_K<L\}}m(d)\,dA
    &=\sum_{j=0}^{\infty}\int_{\Omega_j}m(d)\,dA\\
    &\leq C N\sum_{j=0}^{\infty}
       \frac{\log(2^jR/100)}{2^jR/100}\\
       &\leq CN\frac{\log R}{R}.
\end{align*}
Hence \eqref{potmass} is proved.

The final dyadic region remains intersected with \(\{d_K<L\}\).
Only its area is bounded using larger balls; the identity \(d=d_K\)
is used solely where it is known to hold. The resulting constant
is therefore uniform in the outer scale \(L\).
\end{proof}

We next cut the surface along a regular level of $\tau$. The length and geodesic curvature of this level determine the cost of the surface that we will attach there. The same choice also controls the region between this level and the level one unit farther out.

\begin{lemma}\label{collar}
There is a regular value $t_b\in[4R,5R]$ of $\tau|_\Sigma$ for which $\mathcal C=\Sigma\cap\{\tau=t_b\}$ is a finite union of smooth circles and
\begin{equation}\label{collcost}
 \Lambda+\Area\bigl(\Sigma\cap\{t_b<\tau<t_b+1\}\bigr)
 \leq \frac{CNR^2}{\log R},
\end{equation}
Here $\Lambda$ is the sum of the length of $\mathcal C$, the integral $\int_{\mathcal C}|k_g|\,d\ell$, and the number of circles in $\mathcal C$.
The level $\mathcal C$ may be empty.
\end{lemma}
\begin{proof}
Consider the open region
\[
  \mathcal S=\Sigma\cap\{3R<\tau<7R\}.
\]
For \(y\in\mathcal S\), we have
\[
  2R<3R-C_0<d_K(y)<7R+C_0<8R<L.
\]
Consequently \(d(y)=d_K(y)>2R>a\), and
\(\mathcal S\subset\Gamma_a\). In particular, \(\mathcal S\) is a smooth
embedded surface without boundary. At \(y\in\mathcal S\),
put \(\rho=d(y)/8\). The ball \(B_{4\rho}(y)=B_{d(y)/2}(y)\)
contains no point of \(Z_*\), so every zero in that ball satisfies
\(r<\delta_*\). Moreover \(\rho>R/4\geq R_0\).
Proposition 2.6, applied with center \(y\) and scale \(\rho\), yields
\begin{equation}\label{second fundamental eatimate}
  |A(y)|\leq\frac{C}{\rho}=\frac{8C}{d(y)}
    \leq\frac{C_A}{R}.
\end{equation}
Here \(C_A\) is a fixed positive constant.

Write \(f=\tau|_{\mathcal S}\), a smooth function on a smooth surface.
Let \(\mathcal V\subset(4R,5R)\) denote the set
of regular values of $f$. Thus
\[
  |(4R,5R)\setminus\mathcal V|=0,
\]
and for \(t\in\mathcal V\),
\[
  \mathcal C_t=\Sigma\cap\{\tau=t\}=f^{-1}(t)
\]
is a smooth embedded one-dimensional manifold without boundary, or
is empty. It is known (see \cite[Proposition~1.11(d), p.~8]{Lee13}) that a compact one-dimensional manifold has only finitely many connected
components.  At a point of \(\mathcal S\) where
\[
  g=|\nabla_{\mathcal S}f|>0,
\]
choose a unit vector \(T\) tangent to the level of \(f\), the Gauss formula gives
\[
  \Hess_{\mathcal S}f(T,T)
    =D^2\tau(T,T)+A(T,T)\langle\nabla\tau,\nu\rangle.
\]
Thus \eqref{second fundamental eatimate} yields
\[
  |k_g|\,g
    \leq |D^2\tau|+|A|\,|\nabla\tau|
    \leq C_0+\frac{C_A}{R}
    \leq C_H,
  \qquad g\leq|\nabla\tau|\leq1,
\]
where \(C_H\) is fixed.

Let
\[
  \mathcal B=\Sigma\cap\{4R<\tau<5R+1\}.
\]
Then $\mathcal B$  is a subset of \(\mathcal S\). For \(y\in\mathcal B\),
\[
  3R<d_K(y)<5R+1+C_0\leq6R.
\]
Because \(K\) is compact, there is a nearest point \(z\in K\).
Choose \(i\) with \(|z-z_i|<2\), then
\[
  |y-z_i|\leq|y-z|+|z-z_i|<6R+2\leq8R.
\]
It follows that
\[\mathcal B\subset\bigcup_{i=1}^{N}B_{8R}(z_{i})\]
Since \(\mathcal B\subset\Gamma_a\), Lemma 4.1 implies
\begin{align*}
  \Area(\mathcal B)
    &\leq\sum_{i=1}^N
      \Area\bigl(\Gamma_a\cap B_{8R}(z_i)\bigr)\\
    &\leq C N\frac{(8R)^3}{\log(8R)}
     \leq C_B\frac{NR^3}{\log R}.
\end{align*}

Let
\[
  \Omega=\Sigma\cap\{4R<\tau<5R\}\subset\mathcal B,
\]
and for \(t\in\mathcal V\) define
\[
  \ell(t)=\Length(\mathcal C_t),\qquad
  G(t)=\int_{\mathcal C_t}|k_g|\,d\ell.
\]
The coarea formula gives
\[
  \int_{4R}^{5R}\ell(t)\,dt
    =\int_\Omega g\,dA
    \leq\Area(\mathcal B)
    \leq C_B\frac{NR^3}{\log R}.
\]

For the curvature integral, define on \(\Omega\cap\{g>0\}\)
\[
  w(y)=\frac{|\Hess_{\mathcal S}f(T,T)|}{g(y)},
\]
and put \(w=0\) on \(\{g=0\}\). The definition is independent of the
sign of the unit level tangent \(T\). It gives a nonnegative
measurable function, and \(w=|k_g|\) on every regular level.
Furthermore \(wg\leq C_H\) everywhere, with the product taken as
zero at critical points. Weighted coarea therefore yields
\begin{align*}
  \int_{4R}^{5R}G(t)\,dt
    &=\int_\Omega w g\,dA\leq C_H\Area(\mathcal B)
     \leq C_H C_B\frac{NR^3}{\log R}.
\end{align*}
\medskip

For \(t\in(4R,5R)\), put
\[
  H(t)=\Area\bigl(\Sigma\cap\{t<\tau<t+1\}\bigr).
\]
All these strips are contained in \(\mathcal B\). Tonelli's theorem gives
\begin{align*}
  \int_{4R}^{5R}H(t)\,dt
    &=\int_{\mathcal B}
       \left(\int_{4R}^{5R}
          \mathbf 1_{\{t<f(y)<t+1\}}\,dt\right)dA(y)\\
    &=\int_{\mathcal B}
       \bigl|(4R,5R)\cap(f(y)-1,f(y))\bigr|\,dA(y)\\
    &\leq\Area(\mathcal B)
     \leq C_B\frac{NR^3}{\log R}.
\end{align*}
For regular \(t\), let
\[
  F(t)=\ell(t)+G(t)+H(t).
\]
Extend \(\ell\) and \(G\) as above across nonregular values. The three
preceding estimates imply
\[
  \int_{4R}^{5R}F(t)\,dt
    \leq C_1\frac{NR^3}{\log R}.
\]
The interval has length \(R\). By the elementary averaging inequality,
\[
  \left|\left\{t\in(4R,5R):
      F(t)>2C_1\frac{NR^2}{\log R}\right\}\right|
    \leq\frac{R}{2}.
\]
Thus the complementary set has positive measure. Its intersection
with the full-measure set \(\mathcal V\) is nonempty. Choose \(t_b\)
in that intersection. Then \(t_b\) is regular and
\begin{equation}\label{length eatimate}
  \Length(\mathcal C_{t_b})
    +\int_{\mathcal C_{t_b}}|k_g|\,d\ell
    +\Area\bigl(\Sigma\cap\{t_b<\tau<t_b+1\}\bigr)
    \leq 2C_1\frac{NR^2}{\log R}.
\end{equation}

If the selected level is empty, then we are done. If $\mathcal C_{t_b}\neq\emptyset$. Write the finite decomposition
\[
  \mathcal C_{t_b}=\gamma_1\sqcup\cdots\sqcup\gamma_q.
\]
For each circle choose an arclength parametrization $s$, with unit tangent
\(T\). Then
\[
  \frac{dT}{ds}
    =k_g n+A(T,T)\nu
\]
and
\[
  \kappa=\sqrt{k_g^2+A(T,T)^2}
    \leq |k_g|+\frac{C_A}{R}.
\]
By \cite{Fen29}, we know that every smooth regular closed curve in
\(\R^3\) has total Euclidean curvature at least \(2\pi\). Applying
it to every component and summing, we obtain
\[
  2\pi q
    \leq\sum_{j=1}^q\int_{\gamma_j}\kappa\,d\ell
    \leq\int_{\mathcal C_{t_b}}|k_g|\,d\ell
       +\frac{C_A}{R}\Length(\mathcal C_{t_b}).
\]
Since \(R\geq1\), the estimate \eqref{length eatimate} gives
\[
  q\leq C\frac{NR^2}{\log R}.
\]
Adding this to the three quantities already controlled proves \eqref{collcost}.
\end{proof}

We choose a regular outer value $t_o\in[L/2,3L/5]$. We let $Y$ be the union of those components of $\Sigma\cap\{t_b<\tau<t_o\}$ whose closures meet $\mathcal C$. When $\mathcal C$ is empty, we set $Y=\varnothing$. Otherwise the metric completion of $Y$ is a compact smooth surface with inner boundary $\mathcal C$ and possible outer boundary in $\{\tau=t_o\}$. Indeed, the closed zero annulus is compact, all its zeros are regular, and both cutting levels are transverse. It has finitely many components, and each selected component meets an inner circle.

We use the intrinsic distance from the inner boundary,
\[
 t_o-t_b\geq L/2-5R,
 \qquad \rho(y)=\dist_Y(y,\mathcal C).
\]
The inequality $|\nabla_\Sigma\tau|\leq1$ implies $\rho\geq\tau-t_b$. Every path in $Y$ from $\mathcal C$ to the outer boundary consequently has length at least $t_o-t_b$. Since $|A|^2=H_0^2-2K_\Sigma$, inequality \eqref{surfstab} yields a nonnegative potential $W_Y$, equal to a fixed multiple of $m(d)$, such that
\begin{equation}\label{curvform}
 \int_Y\bigl(|\nabla f|^2+\tfrac23 K_\Sigma f^2+W_Yf^2\bigr)\geq0,
 \qquad \int_YW_Y\leq\frac{CN\log R}{R}.
\end{equation}
To see the sign directly, substitution of $|A|^2=H_0^2-2K_\Sigma$ in \eqref{surfstab}, followed by multiplication by $2/3$, gives
\[
 \frac13\int_Y H_0^2f^2
 \leq\int_Y\left(|\nabla f|^2+\frac23K_\Sigma f^2+W_Yf^2\right).
\]
The left side is nonnegative, which proves \eqref{curvform}. This inequality first holds for test functions vanishing near both boundaries and then for the corresponding zero-trace $H^1$ test functions. In particular, the argument allows the actual zero surface to have nonzero mean curvature.

The inner boundary $\mathcal C$ of $Y$ may have several
components, whereas the area estimate in
Proposition~\ref{area} uses intrinsic balls centred at a
single point. The following lemma attaches a connected
surface along these boundary curves and chooses a point $p$
from which every point of $\mathcal C$ can be reached by
a path of uniformly bounded length. Consequently, a region
of $Y$ lying within a prescribed distance of $\mathcal C$
is contained in a ball centred at $p$ with a slightly larger
radius. We construct a smooth metric on the attached surface,
rather than an embedding into $\R^3$. The bound on its total
absolute Gaussian curvature controls the additional curvature
term in the subsequent area estimate.
\begin{lemma}\label{cap}
If $\mathcal C$ is nonempty, a connected orientable surface $\mathcal P$ with a smooth metric can be attached along every component of $\mathcal C$, with smooth agreement of the metrics across each joining circle. The surface $\mathcal P$ has one flat infinite cylindrical end. There are a point $p\in\mathcal P$ and an absolute constant $D_0$ such that every point of $\mathcal C$ can be joined to $p$ by a path in $\mathcal P$ of length at most $D_0$. The finite part of $\mathcal P$ is the complement of the infinite product part of the cylinder, and
\begin{equation}\label{capcost}
 \Area(\text{finite part of }\mathcal P)+\int_{\mathcal P}|K_{\mathcal P}|\,dA
 \leq C\Lambda.
\end{equation}
\end{lemma}

\begin{proof}
Fix \(C_i\) and suppress its index temporarily. Parametrize it by arclength \(\sigma\in\R/\ell\mathbb Z\).  Signed geodesic coordinates in the original surface give
\begin{equation}\label{metric}
  g=dt^2+f(t,\sigma)^2\,d\sigma^2,
  \qquad f(0,\sigma)=1,\qquad
  |f_t(0,\sigma)|=|k_g(\sigma)|.
\end{equation}
Here \(t<0\) is the \(Y\) side and \(t>0\) is the side to be replaced.  Such
coordinates exist in a neighbourhood of the whole circle because the
circle is compact and embedded. For the metric $g$, the Gaussian curvature is
\[
  K=-\frac{f_{tt}}{f}.
\]
Moreover,
\begin{equation}\label{eq:metric-identity}
 dA=f\,dt\,d\sigma,\qquad
  |K|\,dA=|f_{tt}|\,dt\,d\sigma.
\end{equation}

Choose a width \(\delta>0\) small enough that the collars of the finitely
many circles are disjoint and, on \(0\leq t\leq\delta\),
\begin{equation}\label{eq:collar-choice}
  \delta<1,\qquad \frac12<f<\frac32,
  \qquad \delta M\leq1,
  \qquad M:=\sup_{[0,\delta]\times C_i}|K|.
\end{equation}
The identity \(f_{tt}=-Kf\) and \eqref{metric} give, for each fixed \(\sigma\),
\begin{align}
  |f_{tt}(t,\sigma)|&\leq\tfrac32 M,\notag\\
  |f_t(t,\sigma)|&\leq |k_g(\sigma)|+\tfrac32 Mt,\label{eq:collar-derivatives}\\
  |f(t,\sigma)-1|&\leq t|k_g(\sigma)|+\tfrac34 Mt^2.\notag
\end{align}

Fix one smooth nondecreasing function
\begin{equation*}\label{eq:cutoff}
 \chi:\R\longrightarrow[0,1],\qquad
 \chi=0\ \text{on }(-\infty,1/4],\qquad
 \chi=1\ \text{on }[3/4,\infty).
\end{equation*}
Replace \(f\) on a copy of this collar by
\begin{equation}\label{eq:flattening}
  \widetilde f(t,\sigma)
   =1+\bigl(1-\chi(t/\delta)\bigr)(f(t,\sigma)-1).
\end{equation}
This is a convex combination of \(f\) and \(1\), so
\(1/2<\widetilde f<3/2\).  Near \(t=0\) the new metric equals the original
metric; near \(t=\delta\) it is exactly \(dt^2+\,d\sigma^2\).
Thus it agrees to every order with the metric on \(Y\) at its inner end
and has a product circle of circumference \(\ell\) at its outer end.

By \eqref{eq:flattening}, we can get
\[
 \widetilde f_{tt}
 =\bigl(1-\chi(t/\delta)\bigr)f_{tt}
  -\frac2\delta\chi'(t/\delta)f_t
  -\frac1{\delta^2}\chi''(t/\delta)(f-1).
\]
Using \eqref{eq:collar-derivatives} and integrating in \(t\), we obtain
\begin{align*}
 \int_0^\delta|\widetilde f_{tt}|\,dt
 &\leq CM\delta
   +\frac C\delta\int_0^\delta(|k_g|+Mt)\,dt
   +\frac C{\delta^2}\int_0^\delta(t|k_g|+Mt^2)\,dt\\
 &\leq C|k_g(\sigma)|+CM\delta.
\end{align*}
Let \(E_i\) denote the resulting collar.  By
\eqref{eq:metric-identity} and \eqref{eq:collar-choice},
\begin{equation}\label{eq:collar-cost}
  \Area(E_i)\leq\tfrac32\delta_i\ell_i\leq\tfrac32\ell_i,
  \qquad
  \int_{E_i}|K|\,dA\leq Cb_i+C M_i\delta_i\ell_i
                         \leq C(b_i+\ell_i).
\end{equation}
Every point of \(C_i\) reaches the corresponding point on the product
boundary of \(E_i\) along a curve of length \(\delta_i<1\).

Take the unit round sphere and fix a point \(p\) on it.  Remove the
interiors of \(m+1\) pairwise disjoint closed geodesic disks, none containing
\(p\).  Choose their radii \(\beta_j\) so that
\begin{equation}\label{eq:holes}
  \sum_{j=1}^{m+1}\beta_j<\frac1{100},\qquad
  \sin\beta_i\leq r_i:=\frac{\ell_i}{2\pi}\quad(1\leq i\leq m).
\end{equation}
Let \(S_0\) be the remaining punctured
sphere.  It is connected and orientable, and
\begin{equation}\label{eq:sphere-cost}
  \Area(S_0)\leq4\pi,\qquad \int_{S_0}|K|\,dA\leq4\pi.
\end{equation}
Given
\[
  0<\beta<\frac1{100},\qquad r_*:=\sin\beta\leq r_0,
  \qquad h:=\frac\beta2,
\]
consider the annulus \([0,1]\times(\R/2\pi\mathbb Z)\) with metric
\begin{equation}\label{eq:annulus-metric}
  ds^2+\rho(s)^2\,d\vartheta^2.
\end{equation}
Define
\begin{align}
  a(s)&:=\chi(2s),\qquad
  c_\beta(s):=\chi\!\left(\frac{s-1+h}{h}\right),\notag\\
  \rho(s)&:=r_0+(r_*-r_0)a(s)
     +c_\beta(s)\bigl(\sin(\beta+s-1)-r_*\bigr).
     \label{eq:radius}
\end{align}
The first change takes place before \(s=1/2\); the second is confined
to \(I=[1-h,1]\).  In particular, \(a=1\) on \(I\). Outside \(I\), the last term vanishes and \(\rho\) is a convex combination
of \(r_0\) and \(r_*\).  On \(I\),
\[
 \rho(s)=(1-c_\beta(s))r_*+c_\beta(s)\sin(\beta+s-1),
 \qquad \frac\beta2\leq\beta+s-1\leq\beta.
\]
Therefore \(\rho>0\) everywhere and \(\rho\leq r_0\).  Near the two ends,
\begin{equation}\label{eq:matching}
  \rho(s)=r_0\quad\text{near }0,\qquad
  \rho(s)=\sin(\beta+s-1)\quad\text{near}~1.
\end{equation}
The first expression is a product metric of radius \(r_0\).  For the
second, set \(u=s-1\).  The metric becomes
\[
  du^2+\sin^2(\beta+u)\,d\vartheta^2,
\]
which is exactly the spherical metric continued across the boundary
of a geodesic disk of radius \(\beta\).  Thus the metric matches smoothly
at both ends, including all derivatives.

  To estimate curvature, the first
nonconstant term in \eqref{eq:radius} contributes at most
\(C|r_0-r_*|\) to \(\int_0^1|\rho''|\,ds\).
For the last term, put \(g(s)=\sin(\beta+s-1)\).  On \(I\),
\[
 |g-r_*|\leq h,\qquad |g'|\leq1,\qquad |g''|\leq1,
 \qquad |c_\beta'|\leq C/h,\quad |c_\beta''|\leq C/h^2.
\]
Consequently
\begin{align*}
 \int_I\bigl|[c_\beta(g-r_*)]''\bigr|\,ds
 &\leq\int_I\left(
     |c_\beta''|\,|g-r_*|+2|c_\beta'|\,|g'|
                           +|c_\beta|\,|g''|\right)\,ds\leq C.
\end{align*}
Applying \eqref{eq:metric-identity} with angular period \(2\pi\), we
obtain for this annulus \(Q\)
\begin{equation}\label{eq:annulus-cost}
  \Area(Q)\leq2\pi r_0,
  \qquad
  \int_Q|K|\,dA
     =2\pi\int_0^1|\rho''(s)|\,ds\leq C(r_0+1).
\end{equation}
Every curve with \(\vartheta\) fixed has length exactly \(1\).

For \(1\leq i\leq m\), apply this construction with
\(r_0=r_i=\ell_i/(2\pi)\) and \(\beta=\beta_i\).
At the product end of \(E_i\), the change of variable
\(\vartheta=2\pi\sigma/\ell_i\) gives
\(dt^2+r_i^2\,d\vartheta^2\), so it matches the product end of the
annulus \(Q_i\).  Join the other end to the \(i\)th spherical hole.

At the remaining hole, apply the very same formula with
\(\beta=\beta_{m+1}\) and \(r_0=\sin\beta_{m+1}\).  Denote this annulus
by \(Q_{\mathrm{out}}\).  Its cost in \eqref{eq:annulus-cost} is bounded
by a universal constant.  At its product end attach
\begin{equation}\label{eq:cylinder}
 Z=[0,\infty)\times(\R/2\pi\mathbb Z),\qquad
 g_Z=dz^2+\sin^2\beta_{m+1}\,d\vartheta^2.
\end{equation}
This is a smooth attachment and \(K_Z=0\).

Let \(P\) be the surface obtained from \(S_0\), the pieces \(E_i,Q_i\),
the annulus \(Q_{\mathrm{out}}\), and the cylinder \(Z\), with the
identifications just described.  It is connected.  Its boundary is
exactly the original circles \(C_i\), and its only noncompact part is
the single cylindrical end.  Choose the circle identifications to
reverse the induced boundary orientations.  Such choices can be made
independently, using reflections of the round angular coordinates
when necessary, so \(P\) and \(Y\cup_{\mathcal C}P\) are orientable.

All interior joins occur either between exact product metrics or
between identical spherical metric expressions.  At \(\mathcal C\), the construction
copies the original metric on a full neighbourhood on the other side.
Thus every join is smooth.  In particular, there is no singular
curvature contribution at a joining circle.

The finite part is the union of all these pieces except the interior
of \(Z\); include its joining cross-section to make the finite part
compact.  The cross-section has area zero.  By
\eqref{eq:collar-cost}, \eqref{eq:sphere-cost}, and
\eqref{eq:annulus-cost},
\begin{align*}
 \Area(P_{\mathrm{fin}})+\int_{P}|K_{P}|\,dA
 &\leq 8\pi+C
       +C\sum_{i=1}^m(\ell_i+b_i)
       +C\sum_{i=1}^m(\ell_i+1)\leq C\Lambda.
\end{align*}
Here the cylinder has zero curvature, and its infinite area is
excluded from the finite-part area.  The fixed terms are absorbed
because \(m\geq1\), hence \(\Lambda\geq1\).

It remains to check the distance bound.  Given any \(q\in S_0\), take
a shortest spherical arc from \(p\) to \(q\), of length at most \(\pi\).
A disk of radius less than \(\pi/2\) is geodesically convex, so this
arc enters any removed disk in at most one interval.  Replace such
an interval by the shorter arc on that disk's boundary.  This
replacement has length at most
\[
  \pi\sin\beta_j\leq\pi\beta_j.
\]
The closed disks are pairwise disjoint, so the boundary detours
avoid the other holes.  We obtain a path in \(S_0\) with length at most
\begin{equation}\label{eq:sphere-distance}
 \pi+\pi\sum_{j=1}^{m+1}\beta_j<\pi+\frac\pi{100}.
\end{equation}
This holds also when \(q\) lies on a hole boundary.

For \(y\in C_i\), follow its constant-arclength-coordinate curve through
\(E_i\) and then its corresponding constant-angle curve through \(Q_i\).
Their total length is \(\delta_i+1<2\).  They reach a point \(q_i(y)\)
on the sphere, which can be joined to \(p\) using
\eqref{eq:sphere-distance}.  Hence
\[
 \operatorname{dist}_{P}(p,y)
 \leq\pi+\frac\pi{100}+\delta_i+1
 <\pi+\frac\pi{100}+2<6.
\]
This proves the lemma with \(D_0=6\).
\end{proof}

The following proposition bounds the area of the part of $Y$
lying within a prescribed distance of its inner boundary.
Lemma~\ref{cap} places this region inside a ball centred at $p$
on the completed surface.
We first extend \eqref{curvform} to this surface, with additional
terms controlled by the area of the transition region and the
curvature bound in Lemma~\ref{cap}.
We then choose a test function that decreases linearly with
the distance from $p$ and vanishes outside the ball.
The integrated form of Fiala's inequality
\cite[Lemma~2.3, equation~(2.12)]{BC14}, combined with this
quadratic inequality, yields the required area bound.

\begin{proposition}\label{area}
For $0<s<t_o-t_b-D_0$,
\begin{equation}\label{areabound}
 \Area\{y\in Y:\rho(y)<s\}\leq C_RN(1+s)^2,
 \qquad C_R\leq \frac{CR^2}{\log R}.
\end{equation}
\end{proposition}
\begin{proof}
The case \(Y=\varnothing\) is immediate. Assume henceforth that
\(\mathcal C\neq\varnothing\).

Note that the set \(\Sigma\cap\{t_b\leq\tau\leq t_o\}\) is closed and lies
in a bounded neighborhood of the compact set \(K\), so it is compact.
Regularity of its zeros and transversality of the cutting levels
make it a smooth compact surface with boundary. Consequently \(Y\)
has finitely many components and finitely many boundary circles.

Glue \(P\) to \(Y\) along \(\mathcal C\). At each component of
\(\mathcal D\), we adopt the collar construction
from Lemma 4.3; here no uniform bound on its curvature or area is
needed. Call the resulting surface \(M\). It is smooth, orientable,
connected, and without boundary: every selected component of \(Y\)
meets the connected surface \(P\). It is noncompact because \(P\)
already has an infinite cylinder. Outside a compact set, \(M\) is
a finite disjoint union of product cylinders, each of infinite
length. Hopf–Rinow theorem implies that closed bounded subsets of $M$ are compact.

Define the height function
\[
 F(x)=
 \begin{cases}
   0,&x\in P,\\
   \tau(x)-t_b,&x\in Y,\\
   H,&x\text{ in an outer attachment}.
 \end{cases}
\]
It is locally Lipschitz,
and \(|\nabla F|\leq1\) almost everywhere. Since
\(|\nabla_Y\tau|\leq1\), it is globally \(1\)-Lipschitz
for the Riemannian distance. As \(F(p)=0\) and \(F=H\) on
\(\mathcal D\),
\[
 \dist_M(p,\mathcal D)\geq H,
 \qquad
 \overline{B_M(p,S)}\subset U:=\operatorname{int}_M(Y\cup_{\mathcal C}P)
 \quad(0<S<H).
\]
When \(\mathcal D=\varnothing\), the containment is automatic.
The Lipschitz argument permits a path to enter and leave \(P\)
any number of times.

Choose a fixed smooth function \(\vartheta:\R\to[0,\pi/2]\) with
\[
 \vartheta(t)=0\quad(t\leq1/4),\qquad
 \vartheta(t)=\pi/2\quad(t\geq3/4),
\]
and define on \(M\)
\[
 \chi=\sin\vartheta(F),\qquad
 \psi=\cos\vartheta(F),\qquad
 E=|\nabla\chi|^2+|\nabla\psi|^2.
\]
These compositions
are smooth. Because \(H>1\), they are constant in a neighborhood of each join. They also satisfy
\[
 \chi^2+\psi^2=1,\qquad
 E=\vartheta'(\tau-t_b)^2|\nabla_Y\tau|^2\leq C_\vartheta
 \quad\hbox{on }Y.
\]
Here \(C_\vartheta=\|\vartheta'\|_\infty^2\). Moreover \(E=0\)
outside \(B\) and \(\{\psi\neq0\}\cap Y^\circ\subset B\).
On \(P\), \((\chi,\psi)=(0,1)\).
On every outer attachment, \((\chi,\psi)=(1,0)\).

Extend \(W_Y\) by zero to \(M\), writing the extension as
\(\widetilde W\), and set \(K_M^\pm=\max\{\pm K_M,0\}\). Define
\begin{equation}\label{eq:potential}
 W_c:=\chi^2\widetilde W+E+\frac23K_M^-\psi^2\geq0.
\end{equation}
This is a locally bounded measurable potential. On \(P\) it equals
\((2/3)K_P^-\); it vanishes on the flat part of \(P\) and on every
outer attachment.

For \(u\in C_c^\infty(U)\), the function \(\chi u\) is supported
in \(Y^\circ\), so it is admissible for \(Q_Y\). Since \(\chi\nabla\chi+\psi\nabla\psi=0\), expanding the
gradients gives
\[
 |\nabla(\chi u)|^2+|\nabla(\psi u)|^2
       =|\nabla u|^2+Eu^2.
\]
Together with \(K_M=K_M^+-K_M^-\), this yields
\begin{align*}
 &\int_M\left(|\nabla u|^2+\frac23K_Mu^2+W_cu^2\right)\,dA=Q_Y(\chi u)+\int_M|\nabla(\psi u)|^2\,dA
                  +\frac23\int_MK_M^+\psi^2u^2\,dA.
\end{align*}
Note that very term on the right is nonnegative, hence
\begin{equation}\label{eq:extended}
 \int_M\left(|\nabla u|^2+\frac23K_Mu^2+W_cu^2\right)\,dA\geq0,
 \qquad u\in C_c^\infty(U).
\end{equation}
In particular, the negative curvature
is paid for by the last term in \eqref{eq:potential}.

Write \(I:=\int_MW_c\,dA\). The support information and the input
estimates give
\begin{align*}
 I
 &\leq\int_YW_Y\,dA+C_\vartheta\Area(B)
         +\frac23\int_P|K_P|\,dA
         +\frac23\int_B|K_\Sigma|\,dA\\
 &\leq C\left(\frac{N\log R}{R}+\Lambda
                  +(1+R^{-2})\Area(B)\right)\\
 &\leq C\left(\frac{N\log R}{R}+\frac{NR^2}{\log R}\right)
 \leq C\frac{NR^2}{\log R}.
\end{align*}
In this process, we have used \((\log R)^2\leq R^3\) for \(R\geq e\).
Since \(N\geq1\), we also have
\[
 1+I\leq C\frac{NR^2}{\log R}.
\]

Fix \(0<S<H\), we set
\[
 r_p(x)=\dist_M(p,x),\qquad
 u_S(x)=\left(1-\frac{r_p(x)}S\right)_+,
 \qquad V(S)=\Area B_M(p,S).
\]

Claim: the function \(u_S\) is an admissible weak test in
\eqref{eq:extended}.

Indeed, choose \(S<S'<H\). Its support is
contained in the compact ball \(\overline{B_M(p,S)}\subset U\),
and local mollification gives smooth functions supported in
\(B_M(p,S')\) converging to \(u_S\) in \(H^1(M)\).
The coefficients \(K_M\) and \(W_c\) are bounded on this fixed
compact neighborhood. For any bounded coefficient \(V_0\),
\[
 \left|\int_M V_0(u_j^2-u_S^2)\,dA\right|
 \leq\|V_0\|_\infty\|u_j-u_S\|_2
                      (\|u_j\|_2+\|u_S\|_2)\longrightarrow0,
\]
so the inequality passes to the limit.

The distance identity
\(|\nabla r_p|=1\) almost everywhere and the Lipschitz chain rule give
\[
 \int_M|\nabla u_S|^2\,dA=\frac{V(S)}{S^2},
 \qquad \int_M W_cu_S^2\,dA\leq I.
\]
Since $M$ is a complete connected
orientable noncompact surface without boundary, \cite[Lemma~2.3, equation~(2.12)]{BC14} gives
\[
 \int_{B_M(p,S)}K_M\xi(r_p)^2\,dA
 \leq2\pi\xi(0)^2
          -\int_{B_M(p,S)}(\xi^2)''(r_p)\,dA
\]
for \(\xi\in C^1([0,S])\), piecewise \(C^2\), with
\(\xi\geq0\), \(\xi'\leq0\), \(\xi''\geq0\), and \(\xi(S)=0\).
Take \(\xi(t)=1-t/S\) on \([0,S]\). Then
\(\xi(0)=1\) and \((\xi^2)''=2/S^2\), hence
\begin{equation}\label{eq:curvature-cutoff}
 \int_MK_Mu_S^2\,dA\leq2\pi-\frac{2V(S)}{S^2}.
\end{equation}

Substituting these estimates into \eqref{eq:extended} gives the
entire area argument:
\begin{align*}
 0
 &\leq\int_M\left(|\nabla u_S|^2+\frac23K_Mu_S^2
                                      +W_cu_S^2\right)\,dA\\
 &\leq\frac{V(S)}{S^2}
       +\frac23\left(2\pi-\frac{2V(S)}{S^2}\right)+I
  =\frac{4\pi}{3}+I-\frac{V(S)}{3S^2}.
\end{align*}
Therefore
\begin{equation}\label{eq:ball-area}
 V(S)\leq(4\pi+3I)S^2
       \leq C\frac{NR^2}{\log R}S^2,
 \qquad 0<S<H.
\end{equation}

Finally, let \(0<s<H-D_0\) and \(\rho(y)<s\). By the definition
of intrinsic distance, there is a path in \(Y\) from some
\(z\in\mathcal C\) to \(y\) of length less than \(s\).
Prepend the path in \(P\) from \(p\) to \(z\), whose length is at
most \(D_0\). This proves the one-sided inclusion
\[
 \{y\in Y:\rho(y)<s\}\subset B_M(p,s+D_0).
\]
The original metric, and hence the area measure, is unchanged on
\(Y\). Since \(s+D_0<H\), \eqref{eq:ball-area} applies and gives
\begin{align*}
 \Area\{y\in Y:\rho(y)<s\}
 &\leq(4\pi+3I)(s+D_0)^2\leq C\frac{NR^2}{\log R}(1+s)^2.
\end{align*}
Here \(s+D_0\leq\max\{1,D_0\}(1+s)\), and \(D_0\) is fixed.
Choosing \(C_R=CR^2/\log R\) proves \eqref{areabound}.
\end{proof}

The next result constructs a compactly supported
cutoff $\eta$ on the regular zero set.
We choose $\eta$ to equal one at every surface point whose
perpendicular segment contributes to the region where $h$
and $F_\eta$ will be joined. This ensures that $F_\eta=P$
throughout that region.
We place the inner transition closer to $K$, so that its
gradient terms vanish after multiplication by the weight
used in Section~5.
For the outer transition, we use a logarithmic cutoff whose
Dirichlet energy becomes small by the quadratic area bound.
\begin{proposition}\label{capacity}
Let $R\geq D_0$ be sufficiently large, let $T\geq e^2$, and suppose
\begin{equation}\label{scales}
 L\geq100(R+T+1).
\end{equation}
There is a compactly supported $H^1$ function $0\leq\eta\leq1$ on $\Gamma_a$ such that $F_\eta=P$ in a neighbourhood of the whole region $\{R/2\leq\tau\leq3R\}\subset\R^3$. The perpendicular segments from the inner part of $\supp\nabla_\Sigma\eta$ lie in $\{\tau<R\}$. Those from its outer part lie in $\{\tau>2R\}$. Moreover,
\begin{equation}\label{capbound}
 \int_{\Gamma_a\cap\{\tau>R/2\}}
 \left(\frac32|\nabla_\Sigma\eta|^2+Cm(d)\eta^2\right)dA
 \leq\frac{C_RN}{\log T}+\frac{CN\log R}{R}.
\end{equation}
The function $\eta$ has nonnegative smooth approximations that preserve these constant regions and the separation of the gradient images.
\end{proposition}

\begin{proof}
We first construct $\eta$ and estimate the integral of its squared
gradient outside the inner transition. We then check that every
perpendicular segment entering the matching region starts where
$\eta=1$. Finally, we locate the segments coming from the two
transition regions and explain the smooth approximation.

The choices of $L,R,T$ imply
$T+1+D_0<t_o-t_b$. On $Y$ we define
\[
 \eta_Y(y)=
 \begin{cases}
  1,&\rho(y)\leq1,\\
  \log(T/\rho(y))/\log T,&1<\rho(y)<T,\\
  0,&\rho(y)\geq T.
 \end{cases}
\]
Here $\rho$ is distance along $Y$ from its inner boundary
$\mathcal C$. Thus $\eta_Y$ equals one near $\mathcal C$ and
vanishes near the outer boundary, whose distance from $\mathcal C$
is at least $t_o-t_b$.

We now define $\eta$ on the rest of $\Gamma_a$. Where
$\tau\leq t_b$, we choose $\eta$ to be a smooth nondecreasing
function of $\tau/R$, equal to zero for $\tau\leq R/5$ and
one for $\tau\geq R/3$. On $Y$ we set $\eta=\eta_Y$, and
on the other components of $\{t_b<\tau<t_o\}$ we set
$\eta=0$. We also set $\eta=0$ for $\tau\geq t_o$.
The definitions agree near both cutting levels: every component
meeting $\mathcal C$ belongs to $Y$, and $\eta_Y=1$ near
$\mathcal C$, while $\eta_Y=0$ near the outer boundary.
The resulting function belongs to $H^1(\Gamma_a)$ and satisfies
$0\leq\eta\leq1$. This definition also covers compact components
contained in $\{\tau<t_b\}$. If $Y$ is empty, only the inner
construction is needed.

We verify compact support before estimating the gradient.
The inequality $|\nabla_\Sigma\tau|\leq1$ implies
$\tau(y)\leq t_b+\rho(y)$ on $Y$. Since $t_b\leq5R$, we have
\[
 \supp\eta\subset
 \left\{R/5-C_0\leq d_K\leq5R+T+C_0\right\}
 \subset\{a<d_K<L/4\}.
\]
The set $K$ is compact, and $d=d_K$ in this region. Hence
$\supp\eta$ is a compact subset of the regular surface
$\Gamma_a$, separated from its boundary $d=a$. The segment-length
bound used in Lemma~\ref{sumlemma} also implies that the
perpendicular segments from this support remain in $\{d_K<L/2\}$.

Observe that only the part $\eta_Y$ contributes to the gradient integral in
$\{\tau>R/2\}$, since the inner transition ends at $\tau=R/3$.
The intrinsic distance is Lipschitz with
$|\nabla_\Sigma\rho|\leq1$ almost everywhere, so
\[
 |\nabla_\Sigma\eta_Y|
 \leq\frac{1}{\rho\log T}
 \quad\text{almost everywhere on }\{1<\rho<T\}.
\]
Outside this region, $\eta_Y$ is constant and its gradient vanishes
almost everywhere.
To use the area bound, we express
$\rho^{-2}=T^{-2}+2\int_\rho^T s^{-3}\,ds$ on this region
and apply Tonelli's theorem. We obtain
\begin{equation}\label{logcost}
 \begin{split}
 \int_Y|\nabla_\Sigma\eta_Y|^2\,dA
 &\leq\frac{1}{\log^2T}
 \left(
 \frac{\Area\{\rho<T\}}{T^2}
 +2\int_1^T
 \frac{\Area\{\rho<s\}}{s^3}\,ds
 \right)\\
 &\leq\frac{C_RN}{\log^2T}
       \left(1+\int_1^T\frac{ds}{s}\right)
 \leq\frac{C_RN}{\log T}.
 \end{split}
\end{equation}
In the second line we used \eqref{areabound} and
$(1+s)^2\leq4s^2$ for $s\geq1$. We absorb the fixed numerical
factors into $C_R$, which still satisfies
$C_R\leq CR^2/\log R$. The estimate also covers a compact
component on which $\rho<T$ everywhere, since it only uses the
areas of the sets $\{\rho<s\}$.

On the support of $\eta$ in $\{\tau>R/2\}$, we have
$R/100<d_K<L$. Since $0\leq\eta\leq1$, the potential estimate
\eqref{potmass} gives
\[
 \int_{\Gamma_a\cap\{\tau>R/2\}}m(d)\eta^2\,dA
 \leq\frac{CN\log R}{R}.
\]
Together with \eqref{logcost}, this proves \eqref{capbound},
after absorbing the factor $3/2$ into the same type of constant.
The choices of these constants are independent of $T,L,K,N$
and the diameter of $K$.

We next prove $F_\eta=P$ near the matching region. We take
$X$ with $R/2-1<\tau(X)<3R+1$. Since
$|\tau-d_K|\leq C_0$ and $d=d_K$ for $d_K<L$, the distance
$d(X)$ is comparable to $R$ and exceeds $3a$. The estimate for
the possible base points in the proof of Lemma~\ref{sumlemma}
therefore gives
\[
 |X-y|\leq C\log d(X)\leq C\log R
 \qquad\text{whenever }X=E(y,s),\quad |s|\leq2T_y.
\]
The bound $|\nabla\tau|\leq1$ now implies
\[
 R/2-1-C\log R<\tau(y)<3R+1+C\log R.
\]
For sufficiently large $R$, this interval is contained in
$(R/3,t_b)$, because $t_b\geq4R$. Thus every term contributing
to either sum comes from a point where $\eta(y)=1$, and
\[
 F_\eta(X)=\sum_{E(y,s)=X}\eta(y)p_y(s)
          =\sum_{E(y,s)=X}p_y(s)=P(X).
\]
The open region chosen for $X$ contains
$\{R/2\leq\tau\leq3R\}$, as required.

It remains to locate the segments from the two transition regions.
In the inner transition, $R/5\leq\tau(y)\leq R/3$, so
$d(y)=d_K(y)$ is comparable to $R$. Every relevant segment has
$|s|\leq C\log R$, and hence
\[
 \tau(E(y,s))\leq R/3+C\log R<R.
\]
In the outer transition, $y\in Y$ and
$\tau(y)\geq t_b\geq4R$. We have $d(y)=d_K(y)$ on the
support, so the definition of $T_y$ yields
\[
 \tau(E(y,s))
 \geq d_K(y)-C_0-C\log\bigl(a+C d_K(y)\bigr)>2R.
\]
For the last inequality, we use $d_K(y)\geq4R-C_0$ and
$a=R/1000$. Once $R$ is sufficiently large,
$C\log(a+Ct)\leq t/4$ for every $t\geq4R-C_0$.
The preceding lower bound is therefore at least
$3d_K(y)/4-C_0>2R$. This proves both assertions about the
perpendicular segments.

Finally, only the distance-dependent function $\eta_Y$ requires
smoothing. We approximate it on $Y$ by smooth functions between
zero and one, keeping them equal to one for $\rho\leq1/2$
and equal to zero for $\rho\geq T+1$. These regions lie inside
regions where $\eta_Y$ is already constant, and
$T+1<t_o-t_b$. The approximation described in the proof of
Lemma~\ref{branchlem} preserves these values and converges in
$H^1$. We leave the inner cutoff and the zero values on all other
components unchanged. The functions then join smoothly across
$\mathcal C$, and their outer gradients remain inside $Y$.
Their supports lie in one fixed compact subset of $\Gamma_a$,
using the support estimate above with $T$ replaced by $T+1$.
Consequently, the matching identity and both segment-location
estimates continue to hold. With $R,T,L,K$ fixed, strong
$H^1$ convergence also preserves the limiting gradient and
potential integrals.
\end{proof}

\section{Proofs of the classification theorems}\label{finish}
In this section, we complete the proofs of the main results.
We first establish the estimates needed to combine $h$
near $K$ with $F_\eta$ farther away. The first two lemmas
control the terms produced by integration by parts and by
replacing the weight along perpendicular segments.
Proposition~\ref{weighted} then bounds the weighted integral
associated with $F_\eta$. In Proposition~\ref{match}, we
combine the two functions and apply stability to compare
this integral with the negative contribution near $K$.
These estimates lead to the proof of Theorem~\ref{stable}.
We then prove Theorem~\ref{monotone} and
Corollary~\ref{localcurv} using the three-dimensional
classification.

For the estimates preceding the proof of
Theorem~\ref{stable}, we assume that $Z_*$ is nonempty.
We keep $\delta_*$ and the constants in the local estimates
fixed, and fix the parameter in Young's inequality in
Lemma~\ref{mixed}. These choices determine $b_*>0$.
Throughout this section, the constants denoted by $C$
may depend on these fixed choices but are independent of
$R,T,L,K,N$ and the diameter of $K$.
We first choose $R$ sufficiently large and set $a=R/1000$.
The estimates from Section~4 provide a constant
$C_R\leq CR^2/\log R$, independent of $T,L,K,N$ and the
diameter of $K$. With $R$ fixed, we next choose $T\geq e^2$,
and then choose a finite $L\geq100(R+T+1)$.
After these choices, we apply
Proposition~\ref{clusters} to obtain $K$ and $N$.
In the proof of Theorem~\ref{stable}, we will choose $R$
to make the terms depending only on $R$ small, and then
choose $T$ to make $C_R/\log T$ small.
The lower bound $b_*N$ for the absolute value of the
negative contribution and the upper bounds for the errors
all contain the same factor $N$. Thus the required choices
of $R$ and $T$ can be made before $K$ and $N$ are determined.

To proceed, let us choose a smooth nondecreasing function $\chi$ of $\tau/R$, which is equal to zero when $\tau\leq R$ and one when $\tau\geq2R$, with $|\nabla\chi|\leq C/R$. We then set $w=\chi^2$.
Note that the gradients of $\chi$ is supported in $\{R<\tau<2R\}$, and $|\nabla w|+|\nabla\chi|\leq C/R$. We take $\eta$ from Proposition~\ref{capacity} and use $F=F_\eta$, together with $I_\eta,V_p$ from \eqref{products}.

To simplify the notation, let us introduce
\[
B_h(f)
=
h^2\left|\nabla\left(\frac{f}{h}\right)\right|^2
-
hrf^2.
\]
The first term measures the variation of the quotient $f/h$
and is nonnegative. The second term is nonpositive because
$h>0$ and $r\geq0$. Wherever $f=h$, the quotient equals one,
so the first term vanishes and
\[
B_h(h)=-h^3r.
\]

Recall that we have $|\nabla\log h|\leq\sqrt2$ and $0\leq hr\leq1$. The equation for $h$ gives the pointwise identity
\begin{equation}\label{ground}
 |\nabla f|^2+(3v^2-1)f^2
 =B_h(f)+\divg(f^2\nabla\log h).
\end{equation}
Indeed, we have
\[
 \begin{split}
 B_h(f)+\divg(f^2\nabla\log h)
 &=|\nabla f|^2+f^2\bigl(|\nabla\log h|^2+\Delta\log h-hr\bigr)\\
 &=|\nabla f|^2+f^2\left(\frac{\Delta h}{h}-hr\right)
 =|\nabla f|^2+(3v^2-1)f^2.
 \end{split}
\]
The last equality uses $L_vh=-h^2r$. For a compactly supported test on $\R^3$, the divergence integrates to zero and $\int B_h(f)=Q_v(f)\geq0$.

 We need to compare $\int wB_h(F)$ with $\int wI_\eta$,
which we can estimate using the calculations in Section~3.
Equations \eqref{product} and \eqref{ground} show that
these two integrals differ by a term involving $\nabla w$.
This term occurs only in the region where we join $h$ to $F$,
and it is small because both $P-h$ and $\nabla P-\nabla h$
are small there. The following lemma bounds its absolute
value by $CN/\log R+CN/R^2$.

\begin{lemma}\label{current}
The vector-valued function $J_\eta=F(V_p-F\nabla\log h)$ satisfies
\begin{equation}\label{currid}
 B_h(F)=I_\eta+\divg J_\eta,
 \qquad
 \int wB_h(F)=\int wI_\eta-\int\nabla w\cdot J_\eta.
\end{equation}
Moreover,$$
 \left|\int\nabla w\cdot J_\eta\right|
 \leq\frac{CN}{\log R}+\frac{CN}{R^2}.$$
\end{lemma}

\begin{proof}
Comparing \eqref{ground} for $f=F$ with \eqref{product} leads to
\[
 B_h(F)+\divg(F^2\nabla\log h)
 =I_\eta+\divg(FV_p).
\]
This is the first identity in \eqref{currid}. The vector-valued function $J_\eta$ is compactly supported, so integration by parts is valid even though $w$ extends to infinity. On a neighbourhood of $\supp\nabla w$, Proposition~\ref{capacity} implies $F=P$ and $V_p=\nabla P$. We set $e=P-h$ there. Since $\nabla h=h\nabla\log h$, Lemma~\ref{cutmatch} gives
\[
 J_\eta=(h+e)(\nabla e-e\nabla\log h),
 \qquad |J_\eta|\leq CR^{-2}h+CR^{-4}.
\]
Here $|e|+|\nabla e|\leq CR^{-2}$ and $|\nabla\log h|\leq\sqrt2$. The gradient bound $|\nabla w|\leq C/R$ and \eqref{regmass} now yield
\[
 \begin{split}
 \left|\int\nabla w\cdot J_\eta\right|
 &\leq\frac{C}{R^3}\int_{A_R}h+\frac{C}{R^5}|A_R|\\
 &\leq\frac{C}{R^3}\frac{NR^3}{\log R}
       +\frac{C}{R^5}NR^3
 =\frac{CN}{\log R}+\frac{CN}{R^2}.
 \end{split}
\]This finishes the proof.
\end{proof}
The estimate above uses the bounds for both $P-h$ and
$\nabla P-\nabla h$, each of which is at most $CR^{-2}$
in the matching region. Together with the bound for
$\int_{A_R}h$, these estimates make the additional integral
small compared with $N$ as $R$ increases. If the approximation
error contained an extra factor $(\log R)^2$, the same
calculation would instead produce an upper bound of order
$N\log R$, which is too large for the final comparison with
$b_*N$. We therefore use the estimates in
\cite[Propositions~6.1 and~10.1]{WW19}, which provide the
required order $R^{-2}$.

The calculation in Section~3 produces a negative term
involving $|A|^2$ after integration along the segments
perpendicular to the zero surface. To use this calculation
in the weighted integral, we replace $w(E(y,s))$ by $w(y)$.
The latter is independent of $s$ and can therefore be taken
outside the integral along each segment. The following lemma
shows that the absolute value of the total error caused by
this replacement is at most $CN/\log R$.

\begin{lemma}\label{weightex}
Let $\eta$ be the cutoff from Proposition~\ref{capacity},
and let $w=\chi^2$ be the weight defined above.
We can replace $w(E(y,s))$ by $w(y)$ in the curvature
integral with the following error bound:
\begin{equation}\label{weightcost}
 \left|
 \int_{\Gamma_a}\eta(y)^2
 \int
 \bigl(w(E(y,s))-w(y)\bigr)
 H(s)p_y(s)\partial_sp_y(s)J(y,s)\,ds\,dA(y)
 \right|
 \leq \frac{CN}{\log R}.
\end{equation}
\end{lemma}
\begin{proof}
We first locate the points $y$ where the two weights can
differ. We then estimate the error along each perpendicular
segment and use the area bound \eqref{regmass}.

Suppose that $\eta(y)\ne0$ and that
$w(E(y,s))\ne w(y)$ for some $|s|\leq2T_y$.
Since $w$ changes only in $\{R<\tau<2R\}$, the segment
from $y$ to $E(y,s)$ must meet this region.
The bounds $|\nabla\tau|\leq1$ and
$|s|\leq C\log d(y)$ therefore imply
\[
R-C\log d(y)
\leq \tau(y)
\leq 2R+C\log d(y).
\]
On the support of $\eta$, we also have $d=d_K$ and
$|\tau-d_K|\leq C_0$. Consequently,
\[
R-C_0-C\log d(y)
\leq d(y)
\leq 2R+C_0+C\log d(y).
\]
For sufficiently large $R$, the upper bound forces
$d(y)\leq3R$. Indeed, the function $t-C\log t$ is increasing
for $t\geq a=R/1000$, and
\[
3R-C\log(3R)>2R+C_0.
\]
Substitution into the lower bound then yields $d(y)\geq R/2$.
The preceding bounds for $\tau(y)$ also give
\[
R/2<\tau(y)<3R.
\]
Thus every point contributing to the error lies in
$\Sigma\cap A_R$, and its distance $d(y)$ is between
$R/2$ and $3R$.

We now estimate the error along a segment starting at such
a point. Since $|\nabla w|\leq C/R$, we have
\[
|w(E(y,s))-w(y)|\leq \frac{C|s|}{R}.
\]
The bounds for $H(s)$, the exponential decay of $p_y$ and
$\partial_sp_y$, and the uniform bound for $J$ tell us that
\[
\begin{split}
\int |sH(s)p_y(s)\partial_sp_y(s)|J(y,s)\,ds
&\leq
Cd(y)^{-2}
\int_{\mathbb R}|s|(1+|s|)
e^{-2\sqrt2|s|}\,ds\\
&\leq Cd(y)^{-2}.
\end{split}
\]
It follows that
\[
\begin{split}
&\int
|w(E(y,s))-w(y)|
|H(s)p_y(s)\partial_sp_y(s)|J(y,s)\,ds\\
&\qquad\leq \frac{C}{R\,d(y)^2}
\leq \frac{C}{R^3}.
\end{split}
\]
Since $0\leq\eta\leq1$, integration over the possible
starting points and \eqref{regmass} yield
\[
\begin{split}
&\left|
\int_{\Gamma_a}\eta(y)^2
\int
\bigl(w(E(y,s))-w(y)\bigr)
H(s)p_y(s)\partial_sp_y(s)J(y,s)\,ds\,dA(y)
\right|\\
&\qquad\leq
\frac{C}{R^3}\Area(\Sigma\cap A_R)
\leq
\frac{C}{R^3}\frac{NR^3}{\log R}
=
\frac{CN}{\log R}.
\end{split}
\]
This proves \eqref{weightcost} and finishes the proof.
\end{proof}
  We now estimate $\int_{\mathbb R^3}wB_h(F)$ for the cutoff
$\eta$ constructed in Proposition~\ref{capacity}.
The gradient of the inner cutoff contributes zero because
$w$ vanishes along every perpendicular segment starting
in the region where this cutoff varies.
We use \eqref{capbound} to control the contribution from
the outer cutoff. More precisely, we have the following

\begin{proposition}\label{weighted}
For the function $\eta$ and weight $w$ chosen above,
\begin{equation}\label{weightedq}
 \begin{split}
 \int_{\R^3}wB_h(F)
 &\leq\sigma\int_{\Gamma_a}w(y)
 \left(\frac32|\nabla_\Sigma\eta|^2-\frac12|A|^2\eta^2\right)dA\\
 &\quad+C\int_{\Gamma_a\cap\{\tau>R/2\}}m(d)\eta^2\,dA
       +\frac{CN}{\log R}+\frac{CN}{R^2}.
 \end{split}
\end{equation}
In particular,
\begin{equation}\label{totalcost}
 \int wB_h(F)
 \leq\frac{C_RN}{\log T}
 +CN\left(\frac1{\log R}+\frac{\log R}{R}+R^{-2}\right).
\end{equation}
\end{proposition}
\begin{proof}
We first use a smooth approximation of $\eta$ from
Proposition~\ref{capacity}, keeping the notation $\eta$.
We estimate the terms from one zero surface, then the terms
involving two different surfaces, and finally apply the cutoff
estimate. Lemma~\ref{current} gives the starting inequality
\[
\int_{\mathbb R^3}wB_h(F)
\leq
\int_{\mathbb R^3}wI_\eta
+\frac{CN}{\log R}+\frac{CN}{R^2}.
\]
In \eqref{products}, $D_\eta$ contains the terms from one
surface at a time, and $I_\eta-D_\eta$ contains the terms
involving two different surfaces. Thus we estimate the two
integrals on the right-hand side of
\[
\int_{\mathbb R^3}wI_\eta
=\int_{\mathbb R^3}wD_\eta
 +\int_{\mathbb R^3}w(I_\eta-D_\eta).
\]

We begin with the terms in $D_\eta$ containing
$|\nabla_\Sigma\eta|^2$. For each $y$ where
$\nabla_\Sigma\eta(y)\ne0$, Proposition~\ref{capacity}
places the entire perpendicular segment either in $\{w=0\}$
or in $\{w=1\}$. Consequently,
\[
w(E(y,s))=w(y),
\qquad
\Theta_w(y)=w(y).
\]
The weight can therefore be taken outside the integral in $s$.
The estimate for $Jg_s^{-1}$ in the proof of
Proposition~\ref{surface} then gives
\[
\begin{split}
&\int_{\Gamma_a}\int w(E(y,s))p_y(s)^2J(y,s)
 g_s^{-1}(\nabla_\Sigma\eta,\nabla_\Sigma\eta)\,ds\,dA(y)\\
&\qquad\leq
\int_{\Gamma_a}
w(y)\bigl(\sigma+Cd(y)^{-1}\bigr)
|\nabla_\Sigma\eta(y)|^2\,dA(y).
\end{split}
\]

The remaining terms in $D_\eta$ contain $\eta(y)^2p_yL_vp_y$.
Formula \eqref{normalop} separates $L_vp_y$ into three terms.
The first is $3(v(E(y,s))^2-Q_y(s)^2)p_y$.
Lemma~\ref{diag} gives $|v(E(y,s))|\leq|Q_y(s)|$.
Since $w\geq0$ and $J>0$, we have
\[
3w(E(y,s))\eta(y)^2
\bigl(v(E(y,s))^2-Q_y(s)^2\bigr)
p_y(s)^2J(y,s)\leq0.
\]
Together with the bound for $p_y\mathcal R$ proved in
Lemma~\ref{diag}, this sign gives the explicit upper bound
\[
\begin{split}
&\int_{\Gamma_a}\eta(y)^2
 \int w(E(y,s))p_yL_vp_yJ\,ds\,dA(y)\\
&\qquad\leq
 \int_{\Gamma_a}\eta(y)^2
 \int w(E(y,s))H(s)p_y\partial_sp_yJ\,ds\,dA(y)\\
&\qquad\quad+
 C\int_{\Gamma_a}\Theta_w(y)d(y)^{-4}\eta(y)^2\,dA(y).
\end{split}
\]
Thus the potential term contributes a nonpositive quantity
with its original weight. Only the curvature integral still
requires a comparison between $w(E(y,s))$ and $w(y)$.

Lemma~\ref{weightex} bounds the absolute difference between
these two curvature integrals by $CN/\log R$.
For the integral with weight $w(y)$, the calculation in
Lemma~\ref{diag} gives
\[
\begin{split}
w(y)\eta(y)^2
\int H(s)p_y\partial_sp_yJ\,ds
&=
w(y)\eta(y)^2K_\Sigma(y)\int p_y^2\,ds\\
&\leq
w(y)\eta(y)^2
\left(-\frac{\sigma}{2}|A(y)|^2+Cd(y)^{-4}\right).
\end{split}
\]
Since $w(y)\leq\Theta_w(y)$ and $d^{-4}\leq m(d)$ for
large $d$, the last two estimates imply
\[
\begin{split}
&\int_{\Gamma_a}\eta(y)^2
 \int w(E(y,s))p_yL_vp_yJ\,ds\,dA(y)\\
&\qquad\leq
 -\frac{\sigma}{2}\int_{\Gamma_a}w(y)|A|^2\eta^2\,dA
 +C\int_{\Gamma_a}\Theta_w m(d)\eta^2\,dA
 +\frac{CN}{\log R}.
\end{split}
\]
This bounds the remaining contribution to $\int wD_\eta$.

We next apply Lemma~\ref{mixed} to $I_\eta-D_\eta$.
Using $\Theta_w=w(y)$ wherever $\nabla_\Sigma\eta\ne0$,
we obtain
\[
\begin{split}
\int_{\mathbb R^3}w(I_\eta-D_\eta)
&\leq\int_{\Gamma_a}w(y)
 \left(\frac{\sigma}{4}+Cd^{-2}\log d\right)
 |\nabla_\Sigma\eta|^2\,dA\\
&\quad+C\int_{\Gamma_a}\Theta_w m(d)\eta^2\,dA.
\end{split}
\]
As in Proposition~\ref{surface}, the lower bound on $a$
ensures that
\[
Cd^{-1}+Cd^{-2}\log d\leq\frac{\sigma}{4}
\qquad\text{for }d\geq a.
\]
The sum of the two coefficients of
$w(y)|\nabla_\Sigma\eta|^2$ is therefore at most
$\sigma+\sigma/4+\sigma/4=3\sigma/2$.
Adding the estimates for $D_\eta$ and $I_\eta-D_\eta$, we have
\[
\begin{split}
\int_{\mathbb R^3}wI_\eta
&\leq
\sigma\int_{\Gamma_a}w(y)
\left(
\frac32|\nabla_\Sigma\eta|^2
-\frac12|A|^2\eta^2
\right)\,dA\\
&\quad+
C\int_{\Gamma_a}\Theta_w m(d)\eta^2\,dA
+\frac{CN}{\log R}.
\end{split}
\]

We now show that the integral containing $\Theta_w$ receives
contributions only from $\{\tau>R/2\}$.
Suppose that $\eta(y)\ne0$ and $\tau(y)\leq R/2$.
The support property in Proposition~\ref{capacity} gives
\[
d(y)=d_K(y)\leq R/2+C_0.
\]
The definition of $T_y$ then gives $|s|\leq C\log R$
along every contributing perpendicular segment.
Since $|\nabla\tau|\leq1$, we have
\[
\tau(E(y,s))
\leq \tau(y)+|s|
\leq R/2+C\log R<R
\]
for sufficiently large $R$.
Thus $w$ vanishes along the segment and $\Theta_w(y)=0$.
Using also $0\leq\Theta_w\leq1$, we obtain
\[
\int_{\Gamma_a}\Theta_w m(d)\eta^2\,dA
\leq
\int_{\Gamma_a\cap\{\tau>R/2\}}m(d)\eta^2\,dA.
\]
Substitution into the estimate for $\int wI_\eta$ and then
into the first inequality of this proof establishes
\eqref{weightedq} for the smooth approximation.

We next pass to the original $H^1$ cutoff. The approximations
from Proposition~\ref{capacity} converge in $H^1$ and have
supports in a fixed compact subset of $\Gamma_a$.
By \eqref{ampgrad}, the corresponding functions $F_\eta$
converge in $H^1(\mathbb R^3)$.
We have
\[
B_h(F)=|\nabla F-F\nabla\log h|^2-hrF^2.
\]
The bounds $|\nabla\log h|\leq\sqrt2$ and $0\leq hr\leq1$
therefore imply convergence of its weighted integral.
The surface integrals also converge, since their coefficients
are bounded on the fixed compact support.
Thus \eqref{weightedq} holds for the chosen $H^1$ function $\eta$.

To obtain \eqref{totalcost}, we use $w|A|^2\eta^2\geq0$,
$0\leq w\leq1$, and $w=0$ on $\{\tau\leq R\}$.
These facts give
\[
\begin{split}
\sigma\int_{\Gamma_a}w(y)
 \left(\frac32|\nabla_\Sigma\eta|^2-\frac12|A|^2\eta^2\right)dA
&\leq\frac{3\sigma}{2}
 \int_{\Gamma_a}w(y)|\nabla_\Sigma\eta|^2\,dA\\
&\leq\frac{3\sigma}{2}
 \int_{\Gamma_a\cap\{\tau>R/2\}}
|\nabla_\Sigma\eta|^2\,dA.
\end{split}
\]
Consequently, \eqref{weightedq} implies
\[
\begin{split}
\int_{\mathbb R^3}wB_h(F)
&\leq\frac{3\sigma}{2}
 \int_{\Gamma_a\cap\{\tau>R/2\}}|\nabla_\Sigma\eta|^2\,dA
 +C\int_{\Gamma_a\cap\{\tau>R/2\}}m(d)\eta^2\,dA\\
&\quad+\frac{CN}{\log R}+\frac{CN}{R^2}.
\end{split}
\]
The cutoff estimate \eqref{capbound} bounds the sum of the
first two integrals by $C_RN/\log T+CN\log R/R$, after
including the fixed factor $\sigma$ in the constants.
This proves \eqref{totalcost}, with the same dependence of
the constants on the parameters as in \eqref{capbound}.
\end{proof}

Proposition~\ref{weighted} bounds
$\int_{\R^3}wB_h(F_\eta)$ from above.
We would also like to obtain a lower bound for the same integral by joining
$h$ near $K$ to $F_\eta$ farther away.
Since $B_h(h)=-h^3r$, the region near $K$ contributes
at most $-b_*N$ to the quadratic form expressed through $B_h$.
The following proposition shows that the additional terms
produced by joining the two functions have total absolute
value at most $CN/R^2$.
Stability requires the full quadratic form to be nonnegative,
so the contribution $\int_{\R^3}wB_h(F_\eta)$ must compensate
for this negative contribution up to the joining error.

\begin{proposition}\label{match}
For the cutoff $\eta$ constructed in Proposition~\ref{capacity}
and the weight $w=\chi^2$ chosen above, we have
\begin{equation}\label{matchcost}
b_*N\leq
\int_{\R^3}wB_h(F_\eta)+\frac{CN}{R^2}.
\end{equation}
\end{proposition}

\begin{proof}
We construct a test function that equals $h$ near $K$ and
equals $F_\eta$ outside $\{\tau<2R\}$ by setting
\[
\Phi=(1-\chi)h+\chi F_\eta.
\]
The set $\{\tau\leq2R\}$ is bounded because $K$ is compact.
Together with the compact support of $F_\eta$, this shows
that $\Phi\in H^1_c(\R^3)$.
We first perform the calculations for the nonnegative
smooth approximations of $\eta$ from
Proposition~\ref{capacity}, which preserve the regions
where $\eta$ is constant.

We use the expression
\[
B_h(f)=h^2\left|\nabla\left(\frac{f}{h}\right)\right|^2
-hrf^2
\]
to separate the contribution of $F_\eta$ from the terms
produced by joining the two functions.
Since $\Phi/h=1+\chi(F_\eta/h-1)$, we have
\[
h\nabla\left(\frac{\Phi}{h}\right)
=
\chi h\nabla\left(\frac{F_\eta}{h}\right)
+(F_\eta-h)\nabla\chi.
\]
The expansion of the two squares in $B_h(\Phi)$,
together with the integrated identity \eqref{ground},
therefore yields
\[
\begin{aligned}
Q_v(\Phi)
&=\int_{\R^3}B_h(\Phi)
 =\int_{\R^3}wB_h(F_\eta)
  +\int_{\R^3}(F_\eta-h)^2|\nabla\chi|^2\\
&\quad+2\int_{\R^3}\chi(F_\eta-h)\nabla\chi\cdot
h\nabla\left(\frac{F_\eta}{h}\right)\\
&\quad-\int_{\R^3}hr\bigl[
(1-\chi)^2h^2+2\chi(1-\chi)hF_\eta
\bigr].
\end{aligned}
\]

We first estimate the term in the last line.
It is nonpositive because $h,r,F_\eta\geq0$ and
$0\leq\chi\leq1$.
On $\{\dist(X,K)<1\}$, the bound
$|\tau-d_K|\leq C_0$ implies $\tau<R$ for large $R$.
Thus $\chi=0$ there, and the integrand in the last line,
including the minus sign, equals $-h^3r$.
Proposition~\ref{clusters} therefore bounds this term
from above by $-b_*N$.

We next estimate the two terms containing $\nabla\chi$.
Both occur only in $\{R<\tau<2R\}$.
Proposition~\ref{capacity} ensures that $F_\eta=P$
throughout a neighbourhood of this region.
Here
\[
R-C_0\leq d_K\leq2R+C_0<L,
\]
so $d=d_K>3a$ for sufficiently large $R$.
Lemma~\ref{cutmatch} consequently yields
\[
|F_\eta-h|+|\nabla(F_\eta-h)|\leq CR^{-2}.
\]
Moreover, we have
\[
h\nabla\left(\frac{F_\eta}{h}\right)
=
\nabla(F_\eta-h)-(F_\eta-h)\nabla\log h.
\]
The bound $|\nabla\log h|\leq\sqrt2$ then implies
\[
h\left|\nabla\left(\frac{F_\eta}{h}\right)\right|
\leq CR^{-2}
\]
in the same region.

We also have $|\nabla\chi|\leq C/R$ and, by
\eqref{regmass},
\[
|\{R<\tau<2R\}|\leq CNR^3.
\]
These bounds control the two remaining terms as follows:
\[
\begin{split}
&\left|
2\int_{\R^3}\chi(F_\eta-h)\nabla\chi\cdot
h\nabla\left(\frac{F_\eta}{h}\right)
\right|\\
&\qquad\leq
CR^{-2}R^{-1}R^{-2}\,NR^3
=\frac{CN}{R^2},
\end{split}
\]
and
\[
\int_{\R^3}(F_\eta-h)^2|\nabla\chi|^2
\leq CR^{-4}R^{-2}\,NR^3
=\frac{CN}{R^3}.
\]

We now apply stability to $\Phi$, which yields
$Q_v(\Phi)\geq0$.
Together with the preceding estimates, this gives
\[
0\leq Q_v(\Phi)
\leq
\int_{\R^3}wB_h(F_\eta)-b_*N+\frac{CN}{R^2},
\]
where we have absorbed $CN/R^3$ into $CN/R^2$.
We obtain \eqref{matchcost} by rearranging this inequality.

Finally, the approximations from
Proposition~\ref{capacity} converge in $H^1$ with fixed
compact support.
By \eqref{ampgrad}, the corresponding functions $F_\eta$,
and hence the functions $\Phi$, converge in
$H^1(\R^3)$.
The bounds $|\nabla\log h|\leq\sqrt2$ and $0\leq hr\leq1$
allow passage to the limit in the quadratic-form integrals.
Thus \eqref{matchcost} holds for the original cutoff $\eta$.
\end{proof}

We are ready to prove Theorem~\ref{stable}.

\begin{proof}[Proof of Theorem~\ref{stable}]The
solution satisfies $-1<v<1$. We suppose to the contrary that $Z_*\ne\varnothing$.
We keep all local constants fixed as at the beginning
of this section, including $\delta_*$ and the parameter
in Young's inequality.
In particular, $b_*>0$ is fixed independently of
$R,T,L$, and $K$.

For every sufficiently large $R$, every $T\geq e^2$,
and every finite $L\geq100(R+T+1)$,
Proposition~\ref{clusters} supplies a nonempty compact
set $K$ and a finite number $N\geq1$ of disjoint unit balls.
Here we also require $a=R/1000$ to exceed all lower
thresholds in the preceding estimates.

For each such choice, the corresponding cutoff satisfies
\eqref{totalcost} and \eqref{matchcost}.
We would like to combine these inequalities, absorb the additional
$CN/R^2$ term into the same error bound, and divide by $N$.
We then obtain
\begin{equation}\label{contradict}
b_*\leq
\frac{C_R}{\log T}
+C\left(
\frac{1}{\log R}
+\frac{\log R}{R}
+\frac{1}{R^2}
\right).
\end{equation}
The constant $C$ depends only on the fixed local constants,
and $C_R$ depends in addition on $R$.
Both are independent of $T,L,K$, and $N$.

We now choose suitable parameters to contradict
\eqref{contradict}.
We first take $R$ sufficiently large that the expression
multiplied by $C$ contributes less than $b_*/4$.
With this $R$ fixed, we choose $T$ so that
$C_R/\log T<b_*/4$.
We then choose a finite $L\geq100(R+T+1)$ and select
the corresponding $K$ and $N$.
Inequality \eqref{contradict} applies to these choices,
but its right-hand side is strictly less than $b_*/2$.
This contradicts $b_*>0$. Thus $Z_*=\varnothing$ and the proof is then completed.
\end{proof}
It is well known that the stable solution conjecture in $\mathbb{R}^3$ implies the De Giorgi conjecture in $\mathbb{R}^4$. For completeness, here we give detailed proofs.

Let us consider the two limits of the monotone solution as $x_4\to\pm\infty$. Monotonicity implies stability of the 4d solution, and both limits inherit this stability.

\begin{lemma}\label{limits}
Let $u$ satisfy the assumptions of Theorem~\ref{monotone}. Then the limits
\[
 v^\pm(x')=\lim_{t\to\pm\infty}u(x',t)
\]
are bounded stable entire solutions in $\R^3$. Their extensions $V^\pm(x',x_4)=v^\pm(x')$, independent of $x_4$, satisfy $V^-<u<V^+$ everywhere.
\end{lemma}

\begin{proof}
Monotonicity and boundedness yield the limits and the strict inequalities. The limits solve the Allen-Cahn equation in $\mathbb{R}^3$. The positive function $w=u_{x_4}$ satisfies $L_uw=0$, and integration by parts yields
\[
 Q_u(\phi)=\int w^2|\nabla(\phi/w)|^2\geq0.
\]
The limit of the stability inequality for vertical translates proves stability of the extensions $V^\pm$ in $\R^4$. Here, the compactly supported test function stays fixed during the passage to the limit. We now choose $\phi(x',t)=\psi(x')\zeta(t/S)$, where $\zeta$ is a fixed nonzero smooth compactly supported function. We then have
\[
 0\leq S\left(\int\zeta^2\right)Q_{v^\pm}(\psi)
       +S^{-1}\left(\int|\zeta'|^2\right)
                \left(\int\psi^2\right).
\]
This tells us that the limit 3d solution is stable.
\end{proof}

With Theorem 1.1 proved, we are now ready to prove Theorem~\ref{monotone}.

\begin{proof}[Proof of Theorem~\ref{monotone}]
Lemma~\ref{limits} and Theorem~\ref{stable} show that each of the two limits is one of the constant solutions $-1,1$ or the 1d heteroclinic solution.  They are all global minimizers. We can therefore apply \cite[Theorem~1.3]{JM04} and deduce that the solution $u$ is global minimizer.

Savin's classification \cite[Theorem~2.3]{Sav09}, valid in dimensions at
most seven, now implies that the level sets of $u$ are parallel
hyperplanes. Thus $u$ is one-dimensional.
\end{proof}

The result of Corollary~\ref{localcurv} may have independent interest.

\begin{proof}[Proof of Corollary~\ref{localcurv}]
We fix $\tau\in(0,1)$ and choose $\delta>0$ as in
\cite[Corollary~1.3]{WW19} for the interval $[-\tau,\tau]$,
with the H\"older exponent fixed at $1/2$. Then  for sufficiently large $R$,
every zero $y\in B_{2R}$ admits a unit vector $e$ such that
\[
\sup_{|x|\leq1/\sqrt2}
\left|
v(y+x)-\tanh\left(\frac{e\cdot x}{\sqrt2}\right)
\right|<\delta.
\]
We now rescale the solution to apply the curvature estimates. Let us set
\[
u_\varepsilon(X)=v(2RX),
\qquad
\varepsilon=\frac{1}{2\sqrt2R}.
\]
On $B_1$, the function $u_\varepsilon$ is stable and satisfies
\[
\varepsilon^2\Delta u_\varepsilon
=\frac12\bigl(u_\varepsilon^3-u_\varepsilon\bigr).
\]
 Since $2R\varepsilon=1/\sqrt2$, a ball of radius
$\varepsilon$ in the $X$ coordinates corresponds to a ball
of radius $1/\sqrt2$ in the original coordinates.
The approximation established above therefore verifies
condition~(1.8) of \cite{WW19}  at every zero in
$B_{1-\varepsilon}$. The remaining hypothesis of
\cite[Corollary~1.3]{WW19} is the classification of
complete two-sided stable minimal surfaces in $\R^3$,
which follows from \cite{dCP79,FCS80}.
That corollary, with the conclusions of
\cite[Theorem~1.1]{WW19}, yields
$\nabla u_\varepsilon\ne0$ and
\[
|A|\leq C_\tau,
\qquad
|H|\leq C_\tau\varepsilon
\quad\text{on }B_{1/2}\cap\{|u_\varepsilon|\leq\tau\},
\]
where both curvatures are measured in the $X$ coordinates.

Returning to the original coordinates multiplies lengths
by $2R$ and therefore divides both curvatures by $2R$.
Since $B_{1/2}$ corresponds to $B_R$, we obtain
\[
|A|\leq\frac{C_\tau}{R},
\qquad
|H|\leq\frac{C_\tau}{R^2}
\quad\text{on }B_R\cap\{|v|\leq\tau\}.
\]
The relation
$\nabla_Xu_\varepsilon(X)=2R\nabla v(2RX)$
also implies $\nabla v\ne0$ there.
We then choose $R_\tau$ large enough to ensure both the initial
approximation and the required smallness of $\varepsilon$.
\end{proof}

\section*{\bf Acknowledgements}
Yong Liu was supported by  National Natural Science Foundation of China No. 12471204. Juncheng Wei was supported by National R\&D Program of China (Grant No. 2022YFA1005602), and Hong Kong General Research Fund (No. 14303125) ``On Fujita equation in the critical or supercritical regime''. Kelei Wang was supported by the National Key R\&D Program of China (No. 2022YFA1005602) and the National Natural Science Foundation of China (No. 12425108 and No. 12221001). Yong Wei was supported by the National Key R\&D Program of China (2021YFA1001800) and National Natural Science Foundation of China (No.~12531002).  Ke Wu is supported by the National Natural Science Foundation of China (No. 12401264) and Yunnan Revitalization Talent Support Program.
The authors acknowledge the use of AI tools. The authors assume full responsibility for the mathematical validity, accuracy, and integrity of the proofs presented in the manuscript.


\end{document}